\documentclass[a4paper,dvipdfmx,12pt]{amsart}
\usepackage{amsmath}
\usepackage{amssymb}
\usepackage{amscd}
\usepackage{mathrsfs}
\usepackage{url}
\usepackage[all]{xy}
\usepackage[dvipdfmx]{graphicx}
\usepackage{graphics}
\usepackage{tikz}
\everymath{\displaystyle}
\theoremstyle{plain} 
\newtheorem{theorem}{\indent\sc Theorem}[section]
\newtheorem{lemma}[theorem]{\indent\sc Lemma}
\newtheorem{corollary}[theorem]{\indent\sc Corollary}
\newtheorem{proposition}[theorem]{\indent\sc Proposition}

\theoremstyle{definition} 
\newtheorem{definition}[theorem]{\indent\sc Definition}

\newtheorem{example}[theorem]{\indent\sc Example}

\newtheorem{observation}[theorem]{\indent\sc Observation}

\newtheorem{question}[theorem]{\indent\sc Question}

\newcommand{\del}{\partial}
\newcommand{\delbar}{\overline{\partial}}
\newcommand{\ddbar}{\del\delbar}
\newcommand{\vp}{\varphi}
\newcommand{\ve}{\varepsilon}

\begin{document}

\pagestyle{plain}
\thispagestyle{plain}

\title[Holomorphic foliation associated with a semi-positive class of numerical dimension one]
{Holomorphic foliation associated with a semi-positive class of numerical dimension one}

\author[Takayuki KOIKE]{Takayuki KOIKE$^{1}$}
\address{ 
$^{1}$ Department of Mathematics \\
Graduate School of Science \\
Osaka City University \\
3-3-138 Sugimoto \\
Osaka 558-8585 \\
Japan 
}
\email{tkoike@osaka-cu.ac.jp}
%
\renewcommand{\thefootnote}{\fnsymbol{footnote}}
\renewcommand{\thefootnote}{\arabic{footnote}}

\begin{abstract} 
Let $X$ be a compact K\"ahler manifold and $\alpha$ be a class in the Dolbeault cohomology class of bidegree $(1, 1)$ on $X$. 
When the numerical dimension of $\alpha$ is one and $\alpha$ admits at least two smooth semi-positive representatives, 
we show the existence of a family of real analytic Levi-flat hypersurfaces in $X$ and a holomorphic foliation on a suitable domain of $X$ along whose leaves any semi-positive representative of $\alpha$ is zero. 
As an application, we give the affirmative answer to \cite[Conjecture 2.1]{K2019} on the relation between the semi-positivity of the line bundle $[Y]$ and the analytic structure of a neighborhood of $Y$ for a smooth connected hypersurface $Y$ of $X$. 
\end{abstract}

\maketitle

\section{Introduction}\label{section:introd}

Let $X$ be a connected compact K\"ahler manifold 
and $\alpha\in H^{1, 1}(X, \mathbb{R}) (:= H^{1, 1}(X, \mathbb{C})\cap H^2(X, \mathbb{R}))$ 
be a class such that $\#{\rm SP}(\alpha)>1$ and ${\rm nd}(\alpha)=1$, where 
${\rm SP}(\alpha)$ is the set of all $C^\infty$'ly smooth $d$-closed semi-positive $(1, 1)$-forms on $X$ which represents the class $\alpha$, and 
\[
{\rm nd}(\alpha) := \max\{k\in \{0, 1, 2, \dots, {\rm dim}\,X\}\mid \alpha^{\wedge k}\not=0\ \text{in}\ H^{k, k}(X, \mathbb{C})\}. 
\]
For such a class $\alpha$, we denote by $K_\alpha$ the closed subset of $X$ defined by 
\[
K_\alpha := \bigcap_{\theta\in {\rm SP}(\alpha)}\bigcap_{\psi\in {\rm PSH}^\infty(X, \theta)}\{x\in X\mid (d\psi)_x=0\}, 
\]
where ${\rm PSH}^\infty(X, \theta)$ is the set of all the $\theta$-plurisubharmonic functions of class $C^\infty$ for a $C^\infty$'ly smooth $(1, 1)$-form $\theta$ on $X$: 
i.e. ${\rm PSH}^\infty(X, \theta) := \{\psi\colon X\to \mathbb{R}: C^\infty\mid \theta+\sqrt{-1}\ddbar \psi\geq 0\}$. 
Note that it follows from the $\ddbar$-lemma that ${\rm PSH}^\infty(X, \theta)\supsetneqq  \mathbb{R}$ holds for $\theta\in {\rm SP}(\alpha)$ and $K_\alpha\subsetneqq X$, since $\#{\rm SP}(\alpha)>1$. 
For such $X$ and $\alpha$, we show the following: 
\begin{theorem}\label{thm:main1}
Let $X$ and $\alpha$ be as above. 
Then there uniquely exists a non-singular holomorphic foliation $\mathcal{F}_\alpha$ on $X\setminus K_\alpha$ of complex codimension $1$ such that $i_\mathcal{L}^*\theta\equiv 0$ for any $\theta\in {\rm SP}(\alpha)$ and any leaf $\mathcal{L}$ of $\mathcal{F}_\alpha$, where $i_\mathcal{L}\colon \mathcal{L}\to X$ is the holomorphic immersion. 
\end{theorem}

We investigate how large can $\mathcal{F}_\alpha$ be analytically extended by classifying the connected components of $K_\alpha$ from the view point of the existence of an {\it $\mathcal{F}_\alpha$-adaptive function} in the following sense on a suitable neighborhood: 
We say that a continuous function $h\colon \overline{W}\to [-\infty, +\infty]$ on the closure of a domain (connected open subset) $W$ of $X$ is $\mathcal{F}_\alpha$-adaptive if $h|_W$ is a $\mathbb{R}$-valued non-constant pluriharmonic function, $h|_{W\setminus K_\alpha}$ is $\mathcal{F}_\alpha$-leafwise constant, and the preimage $h^{-1}(\{\textstyle\max_{\overline{W}}h, \textstyle\min_{\overline{W}}h\})$ coincides with the boundary $\del W$ of $W$, 
where the topology of $[-\infty, +\infty]$ is the one such that $[-\infty, +\infty]$ is homeomorphic to the interval $[0, 1]\subset \mathbb{R}$. 

\begin{definition}
A connected component $K'$ of $K_\alpha$ is said to be {\it an essential component} if there does {\bf not} exist a connected open neighborhood $W$ of $K'$ in $X$ such that $W\cap K_\alpha$ is a relatively compact subset of $W$ and that there exists an $\mathcal{F}_\alpha$-adaptive function on $\overline{W}$. The union of all the essential components of $K_\alpha$ is denoted by $K_\alpha^{\rm ess}$.  
\end{definition}

Our second main result is the following: 
\begin{theorem}\label{thm:main2}
Let $X$ and $\alpha$ be as above. 
Then the holomorphic foliation $\mathcal{F}_\alpha$ on $X\setminus K_\alpha$ as in Theorem \ref{thm:main1} can be extended to $X\setminus K_\alpha^{\rm ess}$ as a (maybe singular) holomorphic foliation. Moreover, one of the following holds: 
\begin{description}
\item[Case I] There exists a surjective holomorphic map $\Phi\colon X\to R$ to a compact Riemann surface $R$ and a K\"ahler class $\alpha_R$ of $R$ such that $\alpha = \Phi^*\alpha_R$. 
In this case, $\mathcal{F}_\alpha$ is the foliation defined by the fibration $\Phi$, 
$K_\alpha^{\rm ess}=\emptyset$, $K_\alpha$ is included in the set of all the critical points of $\Phi$, and the set of all the singular points $(\Phi^{-1}(p))_{\rm sing}$ of the (set-theoretical) fiber $\Phi^{-1}(p)$ is included in $K_\alpha$ for any point $p\in R$. 
\item[Case I\!I] Not in the case I and $K_\alpha^{\rm ess}=\emptyset$. 
In this case, there exist an open covering $\{U_1, U_2\}$ of $X$ consisting of two domains and 
$\mathcal{F}_\alpha$-adaptive functions $h_j\colon \overline{U_j}\to [-\infty, +\infty]$ for each $j=1, 2$ such that, on each connected component $W$ of $U_1\cap U_2$, there exist constants $a_W, b_W\in \mathbb{R}$ such that $h_2=a_Wh_1+b_W$ holds on $W$. 
The foliation $\mathcal{F}_\alpha$ is defined on $X$, its tangent bundle is perpendicular to $\del h_j$ on $U_j$, and $K_\alpha\cap U_j=\{x\in U_j\mid (d h_j)_x=0\}$ holds for $j=1, 2$. 
\item[Case I\!I\!I] $K_\alpha^{\rm ess}\not=\emptyset$. 
In this case, $X\setminus K_\alpha^{\rm ess}$ is a domain of $X$ and there exists an $\mathcal{F}_\alpha$-adaptive function $h_\alpha\colon \overline{X\setminus K_\alpha^{\rm ess}}\to [-\infty, +\infty]$. The tangent bundle of the foliation $\mathcal{F}_\alpha$ on $X\setminus K_\alpha^{\rm ess}$ is perpendicular to $\del h_\alpha$, and $K_\alpha\setminus K_\alpha^{\rm ess}=\{x\in X\setminus K_\alpha^{\rm ess}\mid (dh_\alpha)_x=0\}$ holds. 
\end{description}
\end{theorem}

As a corollary, one has the following: 
\begin{theorem}\label{thm:main3}
Let $X$ be a connected compact K\"ahler manifold. 
Assume that there exists a $(1, 1)$-class $\alpha\in H^{1, 1}(X, \mathbb{R})$ with $\#{\rm SP}(\alpha)>1$ and ${\rm nd}(\alpha)=1$. 
Then $X$ admits uncountably many compact Levi-flat hypersurfaces of class $C^\omega$ (i.e. real analytic). 
\end{theorem}

Let $Y$ be a non-singular hypersurface of $X$ 
such that the normal bundle $N_{Y/X}=[Y]|_Y$ is unitary flat (i.e. $N_{Y/X}\in H^1(Y, \mathrm{U}(1))$, where $\mathrm{U}(1):=\{t\in \mathbb{C}\mid |t|=1\}$), where $[Y]$ is the holomorphic line bundle on $X$ which corresponds to the divisor $Y$. 
Note that the first Chern class $c_1([Y])$ of $[Y]$ satisfies ${\rm nd}(c_1([Y]))=1$ in this case. 
Our motivation comes from the study of the relation between {\it the semi-positivity of $[Y]$} (i.e. 
non-emptiness of ${\rm SP}(c_1([Y]))$) and the complex analytic structure of a neighborhood of $Y$. 
In \cite[Conjecture 2.1]{K2019}, we conjectured that $[Y]$ is semi-positive if and only if the pair $(Y, X)$ is of class ($\beta'$) or  ($\beta''$) in the classification of Ueda \cite{U}. 
The following corollary, which follows from \cite[Theorem 1.4]{K2020} and the argument in the proof of Theorem \ref{thm:main3}, gives an affirmative answer to \cite[Conjecture 2.1]{K2019} when $Y$ is non-singular. 
\begin{corollary}\label{cor:main}
Let $X$ be a connected compact K\"ahler manifold and 
$Y\subset X$ be a non-singular connected hypersurface such that $N_{Y/X}$ is unitary flat. 
Then $[Y]$ is semi-positive if and only if there exists a neighborhood $V$ of $Y$ such that $[Y]|_V$ is unitary flat: i.e. there exists a non-singular holomorhic foliation on $V$ which has $Y$ as a leaf along which the holonomy is $\mathrm{U}(1)$-linear. 
\end{corollary}
Note that 
Ohsawa pointed out in \cite[Remark 5.2]{O} that Corollary \ref{cor:main} for a surface $X$ can be shown by combining \cite[Theorem 1.4]{K2020} and Siu's solution \cite{S} of Grauert--Riemenschneider's conjecture (K\"ahlerness assumption is not needed in his proof). 
Note also that this kinds of results can be regarded as a generalization of Brunella's theorem \cite{B} for the blow-up of the projective plane at general nine points. See \S \ref{sec:6} and \S \ref{sec:7} for details. 

The foliation $\mathcal{F}_\alpha$ is constructed by considering the eigenvectors which belongs to the eigenvalue zero of each element of ${\rm SP}(\alpha)$, 
or equivalently, by considering {\it the Monge--Amp\`ere foliation} for each element $\psi\in {\rm PSH}^\infty(X, \theta)$ for an element $\theta\in {\rm SP}(\alpha)$. 
In \S \ref{sec:3}, we will show that $\sqrt{-1}\ddbar \psi = g_\psi\cdot \sqrt{-1}\del\psi\wedge \delbar \psi$ holds for an $\mathcal{F}_\alpha$-leafwise constant function $g_\psi$ on a suitable domain of $X$ essentially by a linear-algebraic arguments. 
When $g_\psi$ is a non-constant function on some level set of $\psi$, we show that the situation is in Case I. 
When $g_\psi$ is constant on any level set of $\psi$, $g_\psi=\chi\circ \psi$ holds on a domain for some real function $\chi$. By considering a solution $G$ of a suitable ordinary differential equation concerning on $\chi$, one can see that $h_0:=G\circ \psi$ is a pluriharmonic function. In this case, we show that the situation is either Case I\!I or I\!I\!I by considering the analytic continuation of $h_0$. 

The organization of the paper is as follows. 
In \S 2, we collect some fundamental facts and known results. 
In \S 3, we investigate the relation between the level sets of a non-constant element of ${\rm PSH}^\infty(X, \theta)$ and Monge--Amp\`ere foliation for an element $\theta\in {\rm SP}(\alpha)$. The main result in this section is Theorem \ref{thm:main_sec3}, in which we classify the situation into two cases $(a)$ and $(b)$ according to the constantness or non-constantness of $g_\psi$ on each level set of $\psi$, where $\psi$ and $g_\psi$ are the functions as above. Here we also show Theorem \ref{thm:main3}. 
In \S 4, we investigate the cases $(a)$ and $(b)$ in Theorem \ref{thm:main_sec3} as a  preliminary step for the proof of main results. 
In \S 5, we show Theorems \ref{thm:main1} and \ref{thm:main2}. 
In \S 6, we investigate the case where the class $\alpha$ is the first Chern class of the line bundle $[Y]$ for a hypersurface $Y$ on $X$. Here we show Corollary \ref{cor:main}. 
In \S 7, we give some examples. 
\vskip3mm
{\bf Acknowledgment. } 
The author would like to give thanks to Prof. Noboru Ogawa for useful discussions. 
He is also grateful to Prof. Valentino Tosatti for his very useful comment, by which the condition that $X$ is either a surface or a projective manifold could be dropped in Theorems \ref{thm:main1} and \ref{thm:main2}. 
The author is supported by JSPS KAKENHI Grant Number 20K14313 and by a program: Leading Initiative for Excellent Young Researchers (LEADER, No. J171000201). 

\section{Preliminaries} \label{sec:preliminaries}

\subsection{On the set ${\rm PSH}^\infty(X, \theta)$}
Let $X$ be a compact K\"ahler manifold and $\theta$ be a $d$-closed $C^\infty$ $(1, 1)$-form on $X$. 

\begin{lemma}\label{lem:psh^infty}
Let $\vp$ and $\psi$ be elements of ${\rm PSH}^\infty(X, \theta)$. \\
$(i)$ Let $\chi\colon \mathbb{R}^2\to \mathbb{R}$ be a function of class $C^\infty$. 
Assume that $\chi$ is non-decreasing in each variable, convex, and that $\chi(p+r, q+r)=\chi(p, q)+r$ for any $p, q, r\in \mathbb{R}$. Then the function $\chi(\vp, \psi)\colon x\mapsto \chi(\vp(x), \psi(x))$ is also an element of ${\rm PSH}^\infty(X, \theta)$. \\
$(ii)$ The function $\log(e^\vp+e^\psi)\colon x\mapsto \log(e^{\vp(x)}+e^{\psi(x)})$ is an element of ${\rm PSH}^\infty(X, \theta)$. 
\end{lemma}
See also \cite[Proposition 2.3]{GZ} for the assertion $(ii)$. 
Note that, in \cite{GZ}, this assertion is proven directly by using the formulae
\begin{equation}\label{eq:del_logexp}
\del \log(e^\vp+e^\psi)
= \frac{e^\vp\del \vp + e^\psi \del\psi}{e^\vp+e^\psi}
\end{equation}
and 
\begin{equation}\label{eq:ddbar_logexp}
\sqrt{-1}\ddbar \log(e^\vp+e^\psi)
= \frac{e^\vp\sqrt{-1}\ddbar\vp+e^\psi\sqrt{-1}\ddbar\psi}{e^\vp+e^\psi}
+ \frac{e^{\vp+\psi}\sqrt{-1}\del(\vp-\psi)\wedge \delbar (\vp-\psi)}{(e^\vp+e^\psi)^2}. 
\end{equation}

\begin{proof}[Proof of Lemma \ref{lem:psh^infty}]
$(i)$ It is enough to show that, for each point $x\in X$, there exists a neighborhood $B$ of $x$ such that $\chi(\vp, \psi)|_B$ satisfies $\theta|_B+\sqrt{-1}\ddbar \chi(\vp, \psi)|_B\geq 0$. 
Take $B$ so that it is a small open ball in a coordinates system around $x$. 
Then, as $H^1(B, \mathcal{O}_B)=0$ by Oka's vanishing, one can easily deduce the existence of a function $f$ on $B$ such that $\theta|_B=\sqrt{-1}\ddbar f$. 
As $\vp$ and $\psi$ are elements of ${\rm PSH}^\infty(X, \theta)$, both $f+\vp$ and $f+\psi$ are plurisubharmonic on $B$. Therefore, by the assumption, $f+\chi(\vp, \psi) = \chi(f+\vp, f+\psi)$ is also plurisubharmonic on $B$, which proves the assertion. \\
$(ii)$ The assertion follows by considering $\chi(s, t):=\log(e^s+e^t)$. 
\end{proof}

Next, let us see some fundamental properties of the set ${\rm PSH}^\infty(X, \theta)$ when $\theta \in {\rm SP}(\alpha)$ for a class $\alpha\in H^{1, 1}(X, \mathbb{R})$ with ${\rm nd}(\alpha)=1$. 

\begin{lemma}\label{lem:2.3}
Let $X$ be a compact K\"ahler manifold of dimension $n$ and $\alpha\in H^{1, 1}(X, \mathbb{R})$ be a class such that ${\rm nd}(\alpha)=1$. 
Take $\theta\in {\rm SP}(\alpha)$. \\
$(i)$ For $\vp, \psi\in {\rm PSH}^\infty(X, \theta)$, it holds that $(\theta+\sqrt{-1}\ddbar \vp)\wedge (\theta+\sqrt{-1}\ddbar \psi)\equiv 0$. \\
$(ii)$ For $\vp\in {\rm PSH}^\infty(X, \theta)$, $\theta \wedge \ddbar \vp\equiv 0$ and $(\theta+\sqrt{-1}\ddbar \vp)^{2}:=(\theta+\sqrt{-1}\ddbar \vp)^{\wedge 2}\equiv 0$ hold. \\
$(iii)$ For $\vp, \psi\in {\rm PSH}^\infty(X, \theta)$, it holds that $\ddbar \vp\wedge\ddbar \psi\equiv 0$. 
\end{lemma}

\begin{proof}
$(i)$ Take a K\"ahler form $\omega$ of $X$. As 
\[
\int_X(\theta+\sqrt{-1}\ddbar \vp)\wedge (\theta+\sqrt{-1}\ddbar \psi)\wedge \omega^{n-2}=0
\]
by the assumption ${\rm nd}(\alpha)=1$, it follows from the semi-positivity of the forms $\theta+\sqrt{-1}\ddbar \vp$ and $\theta+\sqrt{-1}\ddbar \psi$ that $(\theta+\sqrt{-1}\ddbar \vp)\wedge (\theta+\sqrt{-1}\ddbar \psi)\equiv 0$ (More precisely, take a suitable coordinates $(z^1, z^2, \dots, z^n)$ around each point such that both $\omega$ and $\theta+\sqrt{-1}\ddbar \vp$ are represented by diagonal matrices at the point and apply Lemma \ref{lem:lin_alg}). \\
$(ii)$ Note that $\theta^2\equiv 0$ follows from the assertion $(i)$ (consider the case where $\vp\equiv \psi\equiv 0$). 
The assertion $(ii)$ follows from $(i)$ by considering the cases $(\vp, \psi)=(0, \vp)$ and $(\vp, \psi)=(\vp, \vp)$. \\
$(iii)$ The assertion $(iii)$ follows from $(i)$ and $(ii)$. 
\end{proof}

\subsection{On the set $K_\alpha$}

In this subsection, we show the following: 
\begin{lemma}\label{lem:K_alpha}
Let $X$ be a compact K\"ahler manifold and $\alpha\in H^{1, 1}(X, \mathbb{R})$ be a class such that ${\rm nd}(\alpha)=1$ and $\#{\rm SP}(\alpha)>1$. 
Then, for any $\theta\in {\rm SP}(\alpha)$, it follows that 
\[
K_\alpha = \bigcap_{\psi\in {\rm PSH}^\infty(X, \theta)}\{x\in X\mid (d\psi)_x=0\}. 
\]
\end{lemma}

\begin{proof}
Denote by $K_\theta$ the set in the right hand side. 
As the inclusion $K_\alpha\subset K_\theta$ is clear by definition, here we show the other inclusion. 
Take $x\in K_\theta$, $\theta'\in {\rm SP}(\alpha)$, and $\psi\in {\rm PSH}^\infty(X, \theta')$. 
By the $\ddbar$-lemma, there exists a function $f$ on $X$ such that $\theta'=\theta+\sqrt{-1}\ddbar f$. 
By the regularity theorem, one has that $f$ is of class $C^\infty$. 
Thus both the functions $f$ and $f+\psi$ are elements of ${\rm PSH}^\infty(X, \theta)$. 
As $x$ is an element of $K_\theta$, one has that $(d\psi)_x=(d(f+\psi))_x-(df)_x=0$, 
from which it follows that $x\in K_\alpha$, since $\theta'$ and $\psi$ are arbitrary. 
\end{proof}

\subsection{Fundamental observation on $\mathcal{F}_\alpha$-adaptive function}\label{sec:fund_obs_F-adaptive}

Let $X$ be a compact complex manifold of dimension $n$. 
In this subsection, we consider a domain $W$ of $X$ and a continuous function $h\colon \overline{W}\to [-\infty, +\infty]$ such that 
$h|_W$ is a non-constant pluriharmonic function and that $h^{-1}(\{\textstyle\max_{\overline{W}}h, \textstyle\min_{\overline{W}}h\})=\del W$. 
Note that $h$ is $\mathcal{F}_\alpha$-adaptive in the sense of \S 1 if ($\mathcal{F}_\alpha$ as in Theorem \ref{thm:main1} actually exists and) it is $\mathcal{F}_\alpha$-leafwise constant. 
Note also that the image $h(W)$ coincides with the open interval $(\textstyle\min_{\overline{W}}h, \textstyle\max_{\overline{W}}h)$. 
First let us show the following: 
\begin{lemma}
The function $h|_W\colon W\to (\textstyle\min_{\overline{W}}h, \textstyle\max_{\overline{W}}h)$ is proper. 
\end{lemma}

\begin{proof}
It is sufficient to show that the preimage $K := h^{-1}([a, b])$ of the closed interval $[a, b]\subset (\textstyle\min_{\overline{W}}h, \textstyle\max_{\overline{W}}h)$ is compact, since a closed subset of a compact set is also compact (Consider $a:=\min I$ and $b:=\max I$ for a compact subset $I\subset (\textstyle\min_{\overline{W}}h, \textstyle\max_{\overline{W}}h)$). 
As $h^{-1}(\{\textstyle\max_{\overline{W}}h, \textstyle\min_{\overline{W}}h\})=\del W$, it is clear that $\overline{K}\cap \del W=\emptyset$, where the closure is taken by regarding $K$ as a subset of $X$. From this and the closedness of $K$ as a subset of $W$, it follows that $K$ is a closed subset of a compact space $X$. Therefore $K$ is compact. 
\end{proof}

By the real-analyticity of a pluriharmonic function, it follows from \cite{SS} that the set $\mathcal{C}$ of all the critical values of $h|_W$ is discrete. 
Note that $\#\mathcal{C}<\infty$ holds for example if there exists real numbers $a$ and $b$ such that $\textstyle\min_{\overline{W}}h < a < b < \textstyle\max_{\overline{W}}h$ and that $h$ is submersion on $h^{-1}((\textstyle\min_{\overline{W}}h, a))\cup h^{-1}((b, \textstyle\max_{\overline{W}}h))$. 

Sort the elements of $\mathcal{C}$ as $\mathcal{C}=\{c_j\}_{j=N}^M$ ($N\in \mathbb{Z}\cup\{-\infty\}$, $M\in \mathbb{Z}\cup\{+\infty\}$) so that $c_j<c_k$ holds if $j<k$. 
For the interval $I_j:=(c_j, c_{j+1})$, it follows from the properness of $h|_{h^{-1}(I_j)}$ and Ehresmann's lemma that $h^{-1}(I_j)$ is homeomorphic to the product $h^{-1}(r_j)\times I_j$, 
where $r_j\in I_j$ is a regular value. 
From this, one has that $h^{-1}(I_j)$ is connected if and only if each fiber of $h|_{h^{-1}(I_j)}$ is connected. 
Note that the fiber $h^{-1}(r_j)$ is a compact Levi-flat hypersurface of $X$. 

In this paper, we often consider the holomorphic foliation $\mathcal{F}$ on $W$ defined by $T_{\mathcal{F}}=(\del h)^\perp \subset T_W$ for the pair $(W, h)$, where $T_{\mathcal{F}}$ is the holomorphic tangent bundle of $\mathcal{F}$ and $T_W$ is the holomorphic tangent bundle of $W$. 
In other word, $\mathcal{F}$ is the foliation which is defined by letting $\{F=\text{constant}\}$ be the local defining functions of leaves around a point $x\in W$, where $F$ is a holomorphic function defined on a neighborhood of $x$ whose real part ${\rm Re}\,F$ coincides with $h$ (Note that $dF=2\del h$). We regard the set $\mathcal{F}_{\rm sing}:=\{x\in W\mid (dh)_x=0\}$ as the singular part of the foliation $\mathcal{F}$ (though $\mathcal{F}$ may be regarded as a non-singular holomorphic foliation not only on $X\setminus \mathcal{F}_{\rm sing}$ but also on the regular part of an irreducible component of $\mathcal{F}_{\rm sing}$ of codimension 1). 
For this foliation $\mathcal{F}$, we show the following: 
\begin{lemma}\label{lem:fol_sing_lwconst}
Let $W, h$, and $\mathcal{F}$ be as above, 
and $\vp\colon W\to \mathbb{R}$ be a function of class $C^\infty$ such that $\vp|_{W\setminus \mathcal{F}_{\rm sing}}$ is $\mathcal{F}$-leafwise constant. 
Then $\vp$ is $\mathcal{F}$-leafwise constant also on $W$ in the following sense: 
For each $x_0\in \mathcal{F}_{\rm sing}$ and a holomorphic function $F\colon B\to \mathbb{C}$ on a sufficiently small neighborhood $B$ of $x_0$ such that ${\rm Re}\,F=h|_B$, there exists a constant $w\in \mathbb{C}$ such that $B\cap \mathcal{F}_{\rm sing}$ is included in the preimage $F^{-1}(w)$ and $\vp|_{F^{-1}(w)}$ is constant. 
Moreover, for $x_0, B$ and $F$ as above, the following holds if $x_0$ is a regular point of a irreducible component of $B\cap \mathcal{F}_{\rm sing}$ of codimension $1$ by shrinking $B$ if necessary: 
There exists a coordinate system $w=(w^1, w^2, \dots, w^n)$ of $B$ with $x_0=(0, 0, \dots, 0)$ and an integer $m\in \mathbb{Z}_{>1}$ such that $F(w)=(w^1)^m+F(x_0)$ and that $\vp(w)=\chi(w^1)$ holds on $B$ for some function $\chi\colon \vp(B)\to \mathbb{R}$. 
\end{lemma}

\begin{proof}
Let $x_0, B$, and $F$ be as in the statement. 
By changing $h$ and $F$ by adding some constant, we may assume $h(x_0)=F(x_0)=0$ without loss of generality. 
Set $S:= B\cap \mathcal{F}_{\rm sing}$ and let $S=\textstyle\bigcup_{\mu=1}^M S_\mu$ be the irreducible decomposition of $S$. 
By shrinking $B$ if necessary, we may assume that each $S_\mu$ includes the point $x_0$. 
It holds that $d(F|_{(S_\mu)_{\rm reg}}) = (dF)|_{(S_\mu)_{\rm reg}}\equiv 0$ holds on the regular part $(S_\mu)_{\rm reg}$ of $S_\mu$ for each $S_\mu$ by the definition of $S$. 
Thus if follows that $F|_{S_\mu}$ is a constant function, since there exists an open dense connected subset of $S_\mu$ which is included in $(S_\mu)_{\rm reg}$ by local parametrization theorem \cite[4.19]{D}. 
Therefore, one has that $F|_S\equiv 0$ since $F(x_0)=0$: i.e. $S\subset F^{-1}(0)$. 

Next let us consider the analytic set $A:=F^{-1}(0)$. 
Let $A=\textstyle\bigcup_{\nu=1}^NA_\nu$ be the irreducible decomposition. 
We may assume that each $A_\nu$ includes the point $x_0$ again by shrinking $B$ if necessary. 
Take a point $y_0\in (A_\nu)_{\rm reg}$. 
As $A_\nu$ is a hypersurface, one can take holomorphic coordinates $(w^1, w^2, \dots, w^n)$ of a small neighborhood $B'$ of $y_0$ in $B$ such that $y_0=(0, 0, \dots, 0)$ and $B'\cap A=B'\cap A_\nu=\{w^1=0\}$. 
We may assume that there exists $m\in \mathbb{Z}_{>0}$ such that $F|_{B'}=(w^1)^m$. 
By considering the continuous function 
\[
B'\ni y\mapsto \max\left\{\left|\frac{\del\vp}{\del w^2}(y)\right|,\ \left|\frac{\del\vp}{\del w^3}(y)\right|,\ \dots,\ \left|\frac{\del\vp}{\del w^n}(y)\right|\right\}\in \mathbb{R}, 
\]
it follows from the $\mathcal{F}$-leafwise constantness of $\vp|_{B'\setminus A}$ that $\vp|_{B'\cap A_\nu}$ is constant. 
Therefore, again from local parametrization theorem \cite[4.19]{D}, one can deduce that $\vp|_A$ is constant, which proves the first half half of the statement. 

Finally we show the latter half of the statement. 
In this case, we may assume that $M=N=1$, $A_1=S_1$, and that $x_0=y_0$ for $y_0$ as above. 
We also may assume that $B=B'$. Take the coordinates $(w^1, w^2, \dots, w^n)$ as above. 
As $2\del h = dF = m(w^1)^{m-1}\cdot dw^1$, it follows that $m>1$. 
The existence of a function $\chi$ as in the statement follows from the fact that each leaf of $\mathcal{F}|_{B}$ is defined by $\{w^1=\text{constant}\}$. 
\end{proof}


\section{Level sets of a non-constant element of ${\rm PSH}^\infty(X, \theta)$ and Monge--Amp\`ere foliation} \label{sec:3}

Let 
$X$ be a connected compact K\"ahler manifold, 
$\alpha\in H^{1, 1}(X, \mathbb{R})$ a class with ${\rm nd}(\alpha)=1$ and $\#{\rm SP}(\alpha)>1$, 
$\theta$ an element of ${\rm SP}(\alpha)$, 
and $\psi$ be an element ${\rm PSH}^\infty(X, \theta)\setminus\mathbb{R}$. 

For an open subset $U\subset X$ such that $(\theta+\sqrt{-1}\ddbar\psi)_x\not=0$ holds for any point $x\in U$, 
denote by $\mathcal{F}(\theta, \psi, U)$ the foliation on $U$ whose tangent space $T_{\mathcal{F}(\theta, \psi, U), x}$ at a point $x\in U$ coincides with the set of all the eigenvectors which belongs to the eigenvalue $0$ of $(\theta+\sqrt{-1}\ddbar\psi)_x$. 
In another word, for a small (simply connected Stein) neighborhood $B$ of each point of $U$, $\mathcal{F}(\theta, \psi, U)|_B$ is the Monge--Amp\`ere foliation of $\sqrt{-1}\ddbar (\vp_0+\psi)$, where $\vp_0$ is a function on $B$ such that $\sqrt{-1}\ddbar \vp_0=\theta|_B$.  Note that each leaf of $\mathcal{F}(\theta, \psi, U)$ is a holomorphically immersed complex submanifolds of $U$ (\cite{So}, see also \cite{BK}, note also that here we used Lemma \ref{lem:2.3} $(ii)$). 
In this section, we investigate the relation between the leaves of $\mathcal{F}(\theta, \psi, U)$ for a suitable $U$ and level sets of $\psi$. 
In what follows, we use the following notation: 
Let $I_\psi$ be the image $\psi(X)$, which is a bounded closed connected interval, 
$(I_\psi)_{\rm reg}$ the set of all the regular values of $\psi$, 
and $U_\psi$ be the preimage $\psi^{-1}((I_\psi)_{\rm reg})$. 
Note that the Lebesgue measure of $I_\psi\setminus (I_\psi)_{\rm reg}$ is zero by Sard's theorem. 

In this section, we often use the following fundamental fact. 
\begin{lemma}\label{lem:lin_alg}
Let $A=(a_{jk})_{j, k=1}^n$ be a Hermitian matrix of order $n$. Assume that $A$ is positive semi-definite and that $a_{jj}=0$ for each $j\in\{2, 3, \dots, n\}$. 
Then it holds that $a_{jk}=0$ for any $j$ and $k$ with $(j, k)\not=(0, 0)$. 
\end{lemma}

\begin{proof}
For $j, k\in \{1, 2, \dots, n\}$ with $j<k$, consider a submatrix
\[
A_{jk}:=
\left(
\begin{array}{cc}
a_{jj} & a_{jk} \\
a_{kj} & a_{kk} \\
\end{array}
\right). 
\]
As $A_{jk}$ is also positive semi-definite, one has that the determinant ${\rm det}\,A_{jk}$ is non-negative. From this and the assumption, it follows that $-|a_{jk}|^2\geq 0$, which proves the lemma. 
\end{proof}

\subsection{Observation on a neighborhood of a point $x_0$ such that $(d\psi)_{x_0}\not=0$ and $(\theta+\sqrt{-1}\ddbar\psi)_{x_0}\not=0$}
In this subsection, we observe on a neighborhood of a point $x_0\in X$ such that $(d\psi)_{x_0}\not=0$ and $(\theta+\sqrt{-1}\ddbar\psi)_{x_0}\not=0$. 
Take a small neighborhood $B_0$ of $x_0$ such that $(d\psi)_{x}\not=0$ and $(\theta+\sqrt{-1}\ddbar\psi)_{x}\not=0$ holds on each $x\in B_0$. 
By shrinking $B_0$ if necessary, we may assume that the leaf $Y_0$ of $\mathcal{F}(\theta, \psi, B_0)$ which contains $x_0$ is a (holomorphically embedded) complex submanifold of $B_0$. 
Again by shrinking $B_0$, one can take coordinates $(z^1, z^2, \dots, z^n)$ of $B_0$ such that $x_0=(0, 0, \dots, 0)$ and $Y_0=\{z^1=0\}$. 
First we show the following: 
\begin{lemma}\label{lem:key_pointwise}
For $x_0\in X$ as above, the following holds: \\
$(i)$ There exists a constant $a\in \mathbb{C}$ such that $(\del\psi)_{x_0}=a(dz^1)_{x_0}$. \\
$(ii)$ There exists a constant $f_{\theta, \psi}(x_0)\in \mathbb{R}_{\geq 0}$ such that $\theta_{x_0}=f_{\theta, \psi}(x_0)\cdot (\sqrt{-1}\del\psi\wedge \delbar\psi)_{x_0}$. \\
$(iii)$ There exists a constant $g_{\psi}(x_0)\in \mathbb{R}$ such that $(\sqrt{-1}\ddbar \psi)_{x_0}=g_{\psi}(x_0)\cdot (\sqrt{-1}\del\psi\wedge \delbar\psi)_{x_0}$. 
\end{lemma}

\begin{proof}
Consider the function $\widehat{\psi} := \log(1+e^\psi)$, which is an element of ${\rm PSH}^\infty(X, \theta)$ by Lemma \ref{lem:psh^infty}. 
It follows from the equation (\ref{eq:ddbar_logexp}) that
\[
\theta+\sqrt{-1}\ddbar \widehat{\psi}
=\frac{1}{1+e^\psi}\cdot \theta
+\frac{e^\psi}{1+e^\psi}\cdot (\theta + \sqrt{-1}\ddbar\psi)
+ \frac{e^{\psi}}{(1+e^\psi)^2}\cdot \sqrt{-1}\del\psi\wedge \delbar \psi. 
\]
From Lemma \ref{lem:2.3} and the semi-positivity of each of the terms of the right hand side of the equation above, it follows that
\begin{equation}\label{eq:3}
\theta_{x_0}\wedge (\sqrt{-1}dz^1\wedge d\overline{z^1})_{x_0}
=(\sqrt{-1}\del\psi\wedge \delbar \psi)_{x_0}\wedge (\sqrt{-1}dz^1\wedge d\overline{z^1})_{x_0} = 0
\end{equation}
holds, since $(\theta+\sqrt{-1}\ddbar\psi)_{x_0}=A\cdot (\sqrt{-1}dz^1\wedge d\overline{z^1})_{x_0}$ for some $A>0$, which follows from the definition of $Y_0$. 
By considering the wedge product of the form $\textstyle\bigwedge_{k\in \{2, 3, \dots, n\}\setminus \{j\}}\sqrt{-1}dz^k\wedge d\overline{z^k}$ and each of the terms of the equation (\ref{eq:3}) for $j=2, 3, \dots, n$, 
one can deduce from Lemma \ref{lem:lin_alg} that both $\theta_{x_0}$ and $(\sqrt{-1}\del\psi\wedge \delbar \psi)_{x_0}$ are perpendicular to the form $(\sqrt{-1}dz^1\wedge d\overline{z^1})_{x_0}$, which proves the assertion $(i)$ and $(ii)$. 
The assertion $(iii)$ follows from $(i)$, $(ii)$, and the equation $(\theta+\sqrt{-1}\ddbar\psi)_{x_0}=A\cdot (\sqrt{-1}dz^1\wedge d\overline{z^1})_{x_0}$. 
\end{proof}

Let $f_{\theta, \psi}\colon B_0\to \mathbb{R}_{\geq 0}$ and 
$g_{\psi}\colon B_0\to \mathbb{R}$ be the functions such that 
\[
\theta|_{B_0}=f_{\theta, \psi}\cdot \sqrt{-1}\del\psi\wedge \delbar\psi
\]
and 
\[
\sqrt{-1}\ddbar \psi|_{B_0}=g_{\psi}\cdot \sqrt{-1}\del\psi\wedge \delbar\psi
\]
hold, whose existence at $x\in B_0$ is assured by Lemma \ref{lem:key_pointwise} for $x_0=x$. 
As 
\[
f_{\theta,\psi} = \frac{{\rm tr}_\omega \theta}{{\rm tr}_\omega \sqrt{-1}\del\psi\wedge \delbar\psi},\ \text{and}\ 
g_{\psi} = \frac{{\rm tr}_\omega \sqrt{-1}\ddbar \psi}{{\rm tr}_\omega\sqrt{-1}\del\psi\wedge \delbar\psi}
\]
holds for a K\"ahler form $\omega$ on $X$, both $f_{\theta,\psi}$ and $g_\psi$ are of class $C^\infty$. 

Next we show the following lemma on the leafwise constantness of some functions on $B_0$. 
\begin{lemma}\label{lem:key_local}
The following holds: \\
$(i)$ For any element $\vp\in {\rm PSH}^\infty(X, \theta)$, 
there exists a function $F_{\vp, \psi}\colon B_0\to \mathbb{C}$ such that 
$\del\vp = F_{\vp, \psi}\cdot \del\psi$ holds on $B_0$ and the function $|F_{\vp, \psi}|$ is $\mathcal{F}(\theta, \psi, B_0)$-leafwise constant. Especially, $\vp|_{B_0}$ is $\mathcal{F}(\theta, \psi, B_0)$-leafwise constant. \\
$(ii)$ Both the functions $f_{\theta,\psi}$ and $g_{\psi}$ are $\mathcal{F}(\theta, \psi, B_0)$-leafwise constant. 
\end{lemma}

\begin{proof}
As our choice of a point $x_0$ is arbitrary, it is sufficient to show the assertions at $x_0$ or along $Y_0$. 

The first half of the assertion $(i)$ follows by applying the same argument as in the proof of Lemma \ref{lem:key_pointwise} to the function $\widehat{\vp}:=\log(1+e^\vp)$ and $(\theta+\sqrt{-1}\ddbar\widehat{\vp})\wedge (\theta+\sqrt{-1}\ddbar\psi)$. 
Note that it also follows from the same argument at the same time that the forms $\theta$ and $\theta+\sqrt{-1}\ddbar \vp$ (and thus $\sqrt{-1}\ddbar \vp$) are pointwisely parallel to the form $\sqrt{-1}dz^1\wedge d\overline{z^1}$ (Note also that the assertion is clear at a point $x$ such that $(\del\vp)_x=0$ by letting $F_{\vp, \psi}(x):=0$). 
From now on, we show the $\mathcal{F}(\theta, \psi, B_0)$-leafwise constantness of the function $|F_{\vp, \psi}|$, where $F_{\vp, \psi}$ is the function as in the statement (Note that the last part of the assertion $(i)$ follows from the first half of the assertion since one obtains from this that $d\vp$ is zero along the leaves of $\mathcal{F}(\theta, \psi, B_0)$). 
As $\sqrt{-1}\del\vp\wedge \delbar\vp=|F_{\vp, \psi}|^2\cdot \sqrt{-1}\del\psi\wedge \delbar \psi$, one has that 
\[
-\del\vp\wedge (\sqrt{-1}\ddbar \vp) = (\del|F_{\vp, \psi}|^2)\wedge \sqrt{-1}\del\psi\wedge \delbar \psi
-|F_{\vp, \psi}|^2\cdot g_\psi\cdot \del\psi\wedge (\sqrt{-1}\del\psi\wedge \delbar\psi). 
\]
As the left hand side is pointwisely parallel to the form $dz^1\wedge\sqrt{-1}dz^1\wedge d\overline{z^1} (\equiv 0)$ and the second term of the right hand side is equal to zero, one has that $(d|F_{\vp, \psi}|^2)\wedge \sqrt{-1}\del\psi\wedge \delbar \psi\equiv 0$, which proves the constantness of the funtion $|F_{\vp, \psi}||_{Y_0}$. 

The assertion $(ii)$ can also be shown by the same argument: for a function $G:=f_{\theta, \psi}$ or $G:=g_\psi$, one has that 
\[
0\equiv \del(\sqrt{-1} G\wedge \del\psi\wedge \delbar\psi)
= \del G\wedge \sqrt{-1} \del\psi\wedge \delbar\psi
-G\cdot g_\psi\cdot \del \psi\wedge \sqrt{-1}\del\psi\wedge \delbar \psi
\]
since both $\theta$ and $\sqrt{-1}\ddbar \psi$ is $\del$-closed. 
As the second term of the right hand side is zero, one has that $G|_{Y_0}$ is constant, which proves the lemma. 
\end{proof}

\subsection{Observation on $U_\psi$}

In this subsection, let us consider the foliation $\mathcal{F}(\theta, \psi)$ on $U_\psi$ which is defined by $\mathcal{F}(\theta, \psi):=\mathcal{F}(\theta, \widehat{\psi}, U_\psi)$ for the function 
$\widehat{\psi} := \log(1+e^\psi)$. 
Note that $\widehat{\psi}\in {\rm PSH}^\infty(X, \theta)$ by Lemma \ref{lem:psh^infty}. 
Note also that it follows from the equation (\ref{eq:del_logexp}) that $(I_{\widehat{\psi}})_{\rm reg} = \{\log(1+e^r)\mid r\in (I_\psi)_{\rm reg}\}$, and 
from the equation (\ref{eq:ddbar_logexp}) that $(\theta+\sqrt{-1}\ddbar \widehat{\psi})_x\not=0$ holds on each point $x\in U_\psi$. 
It also follows from the equation (\ref{eq:ddbar_logexp}) and Lemma \ref{lem:key_pointwise} $(i)$ that 
$T_{\mathcal{F}(\theta, \psi)} = (\del \psi)^\perp \subset T_{U_\psi}$. 
Therefore $\mathcal{F}(\theta, \psi)$ is a $C^\infty$ foliation of real codimension $2$. 
On a neighborhood $U$ of each point $x\in U_\psi$ such that $(\theta+\sqrt{-1}\ddbar \psi)_x\not=0$, 
it holds that $\mathcal{F}(\theta, \psi, U)=\mathcal{F}(\theta, \widehat{\psi}, U)$, since 
\[
\theta+\sqrt{-1}\ddbar \widehat{\psi}
=\left(
\frac{1}{1+e^\psi}\cdot f_{\theta, \psi}
+ \frac{e^\psi}{1+e^\psi}\cdot (f_{\theta, \psi} + g_\psi)
+ \frac{e^{\psi}}{(1+e^\psi)^2}
\right)\cdot \sqrt{-1}\del\psi\wedge \delbar \psi
\]
is pointwisely parallel to $\theta+\sqrt{-1}\ddbar \psi= (f_{\theta, \psi} + g_\psi)\cdot \sqrt{-1}\del\psi\wedge \delbar \psi$ on a neighborhood of such a point $x$, where $f_{\theta, \psi}$ and $g_\psi$ are the functions on a neighborhood of $x$ as in the previous subsection. 

For this foliation $\mathcal{F}(\theta, \psi)$ on $U_\psi$, we show the following: 
\begin{proposition}\label{prop:key}
The following holds: \\
$(i)$ For any $r\in (I_\psi)_{\rm reg}$, $\psi^{-1}(r)$ is the union of some leaves of $\mathcal{F}(\theta, \psi)$.  \\
$(ii)$ For any $\widehat{\theta}\in {\rm SP}(\alpha)$, $\widehat{\theta}$ is zero along each leaf of $\mathcal{F}(\theta, \psi)$: i.e. $i_\mathcal{L}^*\widehat{\theta}\equiv 0$ holds for any leaf $\mathcal{L}$ of $\mathcal{F}(\theta, \psi)$, where $i_\mathcal{L}\colon \mathcal{L}\to X$ is the holomorphic immersion. 
\end{proposition}

\begin{proof}
$(i)$ By Lemma \ref{lem:key_pointwise} $(i)$, the function $\psi$ is $\mathcal{F}(\theta, \psi)$-leafwise constant. Therefore, for each leaf $\mathcal{L}$ of $\mathcal{F}(\theta, \psi)$ and for each $r\in (I_\psi)_{\rm reg}$, either $\mathcal{L}\cap \psi^{-1}(r)=\emptyset$ or $\mathcal{L}\subset \psi^{-1}(r)$. holds. \\
$(ii)$ By the $\ddbar$-lemma, there exits a function $\vp\in{\rm PSH}(X, \theta)$ such that $\widehat{\theta}=\theta + \sqrt{-1}\ddbar \vp$. 
The form $\theta$ is zero along the leaves of $\mathcal{F}(\theta, \psi)$ by Lemma \ref{lem:key_pointwise} $(ii)$ and 
$\vp$ is $\mathcal{F}(\theta, \psi)$-leafwise constant by Lemma \ref{lem:key_local} $(i)$, from which the assertion follows. 
\end{proof}

Now we can state the main result of this section. 
\begin{theorem}\label{thm:main_sec3}
The foliation $\mathcal{F}(\theta, \psi)$ is a non-singular holomorphic foliation on $U_\psi$. Moreover, either $(a)$ or $(b)$ holds: \\
$(a)$ There exits a surjective holomorphic map $\Phi\colon X\to R$ to a compact Riemann surface $R$ whose leaves are connected such that $\mathcal{F}(\theta, \psi)$ coincides with the restriction of the foliation on $X$ whose leaves are the fibers of $\Phi$. 
In this case, $\psi=\psi_R\circ \Phi$ for some function $\psi_R\colon R\to \mathbb{R}$. \\
$(b)$ For any $r\in (I_\psi)_{\rm reg}$ and any connected component $A$ of the preimage $\psi^{-1}(r)$, there does not exist a non-constant $\mathcal{F}(\theta, \psi)$-leafwise constant $\mathbb{R}$-valued function on $A$ of class $C^\infty$. 
In this case, for any connected component $U$ of $U_\psi$, there exists an $\mathcal{F}(\theta, \psi)$-adaptive function $h\colon \overline{U}\to [-\infty, +\infty]$ such that $h|_U=\chi_U\circ \psi|_U$ for some strictly increasing function $\chi_U\colon \psi(U)\to \mathbb{R}$ of class $C^\infty$. 
\end{theorem}

Theorem \ref{thm:main3} follows from this Theorem \ref{thm:main_sec3} as follows: 
\begin{proof}[Proof of Theorem \ref{thm:main3}]
In the case $(a)$ of Theorem \ref{thm:main_sec3}, the preimage of almost all Jordan loops in $R$ of class $C^\omega$ are real analytic compact Levi-flat hypersurfaces of $X$. 
In the case $(b)$ of Theorem \ref{thm:main_sec3}, $h^{-1}(r)$ is a real analytic compact Levi-flat hypersurface of $X$ for any $r\in {\rm Image}\,\chi_U$. 
\end{proof}

\begin{proof}[Proof of Theorem \ref{thm:main_sec3}]
By Lemmata \ref{lem:key_pointwise} and \ref{lem:key_local}, there exist $\mathcal{F}(\theta, \psi)$-leafwise constant functions 
$\widehat{f}:=f_{\theta, \widehat{\psi}}\colon U_\psi\to \mathbb{R}_{\geq 0}$ and 
$\widehat{g}:=g_{\widehat{\psi}}\colon U_\psi\to \mathbb{R}$ of class $C^\infty$ such that 
$\theta=\widehat{f}\cdot \sqrt{-1}\del\widehat{\psi}\wedge \delbar\widehat{\psi}$ 
and 
$\sqrt{-1}\ddbar \widehat{\psi}=\widehat{g}\cdot \sqrt{-1}\del\widehat{\psi}\wedge \delbar\widehat{\psi}$ holds on $U_\psi$. 

First, we consider the case where, for any $r\in (I_\psi)_{\rm reg}$ and any connected component $A$ of the preimage $\psi^{-1}(r)$, there does not exist a non-constant $\mathcal{F}(\theta, \psi)$-leafwise constant $\mathbb{R}$-valued function on $A$ of class $C^\infty$. 
Take a connected component $U$ of $U_\psi$. 
As $\widehat{\psi}|_U$ is a proper submersion, it follows from Ehresmann's lemma that $U\cap \widehat{\psi}^{-1}(r)$ is connected for any $r\in \widehat{\psi}(U)$, from which one has that there exists a function $G_U\colon (a, b)\to \mathbb{R}$ such that $\widehat{g}|_U=G_U\circ \widehat{\psi}$ holds, where $(a, b):=\widehat{\psi}(U)$, by the assumption. 
Take a non-constant function $\chi\colon (a, b)\to \mathbb{R}$ such that 
\[
\chi'(t) = \exp\left(-\int_{\frac{a+b}{2}}^tG_U(s)\,ds\right). 
\]
Then, it follow from the equation $\chi''(t) = -G_U(t)\cdot \chi'(t)$ that $\sqrt{-1}\ddbar h\equiv 0$ holds for the function $h:=\chi\circ \widehat{\psi}$ on $U$, since $\sqrt{-1}\ddbar h = (\chi'\circ\widehat{\psi})\cdot \sqrt{-1}\ddbar \widehat{\psi} + (\chi''\circ\widehat{\psi})\cdot \sqrt{-1}\del\widehat{\psi}\wedge \delbar\widehat{\psi}$. 
As the function function $h\colon \overline{U}\to [-\infty, +\infty]$ is clearly $\mathcal{F}(\theta, \psi)$-adaptive by construction, the assertion $(b)$ holds in this case by letting $\chi_U(t):=\chi(\log(1+e^t))$. 
Note that, as is clear by the pluriharmonicity of $h$, $\mathcal{F}(\theta, \psi)$ is a non-singular holomorphic foliation in this case. 

Next, we consider the case where there exist $r\in (I_\psi)_{\rm reg}$ and a connected component $A$ of the preimage $\psi^{-1}(r)$ such that there exists a non-constant $\mathcal{F}(\theta, \psi)$-leafwise constant $\mathbb{R}$-valued function $\widetilde{g}$ on $A$ of class $C^\infty$. 
In this case, it follows form the following Lemma \ref{lem:fibration} that there exist a leaf $Y$ of $\mathcal{F}(\theta, \psi)$ and a surjective holomorphic map $\Phi\colon X\to R$ to a compact Riemann surface $R$ whose fibers are connected such that $Y$ is a fiber of $\Phi$. 
In what follows, we show the existence of the function $\psi_R$ on $R$ as in the assertion $(a)$. 
Note that, if such a function $\psi_R$ exists, then it follows from $d\psi=\Phi^*d\psi_R$ that $\Phi$ and $\psi_R$ has no critical point on $U_\psi$ and $\Phi(U_\psi)$, respectively, from which one obtains by using Lemma \ref{lem:key_pointwise} $(i)$ that $\mathcal{F}(\theta, \psi)$ coincides with the foliation on $U_\psi$ whose leaves are the fibers of $\Phi$. 

Now it is sufficient to show that $\psi|_Z$ is constant for a fiber $Z$ of $\Phi$. 
Assuming that it is not the case (and thus $d\psi|_{Z_{\rm reg}}\not\equiv 0$), we prove it by contradiction. If $d\psi|_{Z_{\rm reg}}\not\equiv 0$, it follows from the equation (\ref{eq:ddbar_logexp}) that $(\theta+\sqrt{-1}\ddbar \widehat{\psi})|_{Z_{\rm reg}}\not\equiv0$. 
From this and the semi-positivity of $\theta+\sqrt{-1}\ddbar \widehat{\psi}$, one has that 
\[
(\alpha. c_1([Z]). \{\omega\}^{n-2}) = 
\int_Z (\theta+\sqrt{-1}\ddbar \widehat{\psi})|_Z\wedge \omega|_Z^{n-2} > 0
\]
holds for a K\"ahler form $\omega$ of $X$. 
On the other hand, it follows 
\[
(\alpha. c_1([Y]). \{\omega\}^{n-2}) = 
\int_Y \theta|_Y\wedge \omega|_Y^{n-2} = 0
\]
from Proposition \ref{prop:key} $(ii)$, since $Y$ is a leaf of $\mathcal{F}(\theta, \psi)$. 
As both $Y$ and $Z$ are fibers of $\Phi$, these lead to the contradiction. 
\end{proof}

\begin{lemma}\label{lem:fibration}
Assume that there exist $r\in (I_\psi)_{\rm reg}$ and a connected component $A$ of the preimage $\psi^{-1}(r)$ such that there exists a non-constant $\mathcal{F}(\theta, \psi)$-leafwise constant $\mathbb{R}$-valued function $\widetilde{g}$ on $A$ of class $C^\infty$. 
Then there exists a leaf $Y$ of $\mathcal{F}(\theta, \psi)$ and a surjective holomorphic map $\Phi\colon X\to R$ to a compact Riemann surface $R$ whose fibers are connected such that $Y$ is a fiber of $\Phi$. 
\end{lemma}

\begin{proof}
By Sard's theorem, there exists an open non-empty connected interval $J\subset \mathbb{R}$ which is included in the set of all the regular values of the function $\widetilde{g}\colon A\to \mathbb{R}$ (Note that the set of all the critical values of $\widetilde{g}$ is closed since the map $\widetilde{g}$, which is a map from a compact space to a Hausdorff space, is a closed map). 
Take different three points $p_1, p_2$, and $p_3$ from $J$, 
and set $Y_j:=\widetilde{g}^{-1}(p_j)$ for $j=1, 2, 3$. 
As each $p_j$ is a regular value, $Y_j$ is a real submanifold of $A$. 
Therefore, as $\widetilde{g}$ is $\mathcal{F}(\theta, \psi)$-leafwise constant and each leaf of $\mathcal{F}(\theta, \psi)$ is a holomorphically immersed complex submanifold, it follows that $Y_j$ is (holomorphically embedded image of) a complex submanifold of $X$ for $j=1, 2, 3$. 
Moreover, by applying Ehresmann's lemma to the map $\widetilde{g}|_{\widetilde{g}^{-1}(J)}$, one obtains that $Y_2$ is homotopic to $Y_3$, from which it follows that the line bundle $L:=[Y_2-Y_3]$ on $X$ is topologically trivial. Note that $L|_{Y_1}$ is holomorphically trivial by the construction. 

Now we may assume that the restriction map $H^1(X, \mathcal{O}_X)\to H^1(Y_1, \mathcal{O}_{Y_1})$ is injective, since otherwise the natural map $X\to {\rm Alb}(X)/{\rm Alb}(Y_1)$ induced by the Albanese map defines a fibration such that $Y_1$ is a fiber by  \cite[Proposition 2.7]{CLPT} (see also \cite{N}). 
As the natural map ${\rm Pic}^0(X)\to {\rm Pic}^0(Y_1)$ is a finite covering of the image in this case, one has that there exists a positive integer $m$ such that $L^m:=L^{\otimes m}$ is the holomorphically trivial line bundle on $X$. 
Thus one obtains a fibration $X\to \mathbb{P}^1=:R$ to the projective line by considering the complete linear system $|L^m|$ of which $Y_2$ and $Y_3$ are fibers. 
Finally one can modify the fibration so that its fibers are connected by considering a branched finite covering of $R$, which proves the lemma. 
\end{proof}


\section{Observation in Cases $(a)$ and $(b)$ in Theorem \ref{thm:main_sec3}} \label{sec:4}

Let 
$X$ be a connected compact K\"ahler manifold and 
$\alpha\in H^{1, 1}(X, \mathbb{R})$ be a class with ${\rm nd}(\alpha)=1$ and $\#{\rm SP}(\alpha)>1$. 
In this section, we investigate the foliation $\mathcal{F}(\theta, \psi)$ and the domain $U_\psi$ for 
$\theta\in {\rm SP}(\alpha)$ and $\psi\in{\rm PSH}^\infty(X, \theta)\setminus\mathbb{R}$ in the following two cases: First in the case where the assertion $(a)$ of Theorem \ref{thm:main_sec3} holds for some $\theta\in {\rm SP}(\alpha)$ and $\psi\in{\rm PSH}^\infty(X, \theta)\setminus\mathbb{R}$, next in the case where the assertion $(b)$ of Theorem \ref{thm:main_sec3} holds for any $\theta\in {\rm SP}(\alpha)$ and any $\psi\in{\rm PSH}^\infty(X, \theta)\setminus\mathbb{R}$. 

From this section, we often use the following\footnote{We first showed Lemma \ref{lem:hodge_index_thm} only when $X$ is either a surface or a projective manifold by using Hodge index theorem and the hard Lefschetz theorem (see \cite[Theorem 6.33]{V} and \cite[Theorem 6.25]{V} for example). The proof in the general case is taught by Prof. Valentino Tosatti. }: 
\begin{lemma}\label{lem:hodge_index_thm}
Let 
$X$ be a compact complex manifold of dimension $n$, 
$\theta$ an element of ${\rm SP}(\alpha)$, 
$\omega$ a K\"ahler form of $X$, 
and $\widehat{\theta}$ be a $C^\infty$'ly smooth $d$-closed semi-positive $(1, 1)$-form. 
Assume that both $\{\widehat{\theta}\wedge \theta\}$ and $\{\widehat{\theta}\wedge \widehat{\theta}\}$ are zero in $H^{2, 2}(X, \mathbb{C})$ and that 
$\textstyle\int_X\widehat{\theta}\wedge \omega^{n-1}
= \textstyle\int_X\theta\wedge \omega^{n-1}$ holds. 
Then $\widehat{\theta}$ is an element of ${\rm SP}(\alpha)$. 
\end{lemma}

\begin{proof}
It is sufficient to show that $\{\eta\}=0\in H^{1, 1}(X, \mathbb{R})$ for the form $\eta:=\theta-\widehat{\theta}$. 
Consider the bilinear form $Q$ defined by 
\[
Q(\beta, \gamma) := \int_X\beta\wedge \overline{\gamma} \wedge \omega^{n-2}. 
\]
As $\textstyle\int_X\eta\wedge \omega^{n-1}=0$ and $Q(\eta, \eta)=0$ by the assumption, the lemma follows from \cite[Theorem A]{DN} (See also the argument in the end of \cite[Remark 3.4]{TZ}). 
\end{proof}

\subsection{Observation when the assertion $(a)$ holds for some $(\theta, \psi)$}\label{sec:4.1}
First, let us investigate the case where the assertion $(a)$ of Theorem \ref{thm:main_sec3} holds for some $\theta\in {\rm SP}(\alpha)$ and $\psi\in{\rm PSH}^\infty(X, \theta)\setminus\mathbb{R}$. 
Take such $\theta$ and $\psi$. 
Let $\Phi\colon X\to R$ and $\psi_R\colon R\to \mathbb{R}$ be those as in the assertion $(a)$. 
Then one can show the following: 
\begin{proposition}\label{prop:case_a}
Let $X, \alpha, \theta, \psi, \Phi$, and $R$ be as above. 
Then the following holds: \\
$(i)$ There exists a K\"ahler class $\alpha_R$ of $R$ such that $\alpha=\Phi^*\alpha_R$. \\
$(ii)$ Any element of ${\rm SP}(\alpha)$ is zero along each fiber of $\Phi$. \\
$(iii)$ The set $K_\alpha$ is included in the set of all the critical points of $\Phi$. \\
$(iv)$ For any $p\in R$, the set of all the singular points $(\Phi^{-1}(p))_{\rm sing}$ of the (set-theoretical) fiber $\Phi^{-1}(p)$ is included in $K_\alpha$. 
\end{proposition}

\begin{proof}
$(i)$ Fix a K\"ahler form $\omega_R$ of $R$ 
and a non-constant non-negative function $\rho\colon R\to \mathbb{R}_{\geq 0}$ of class $C^\infty$ such that the support ${\rm Supp}\,\rho$ of $\rho$ is included in $\Phi(U_\psi)$. 
Then, as $\widehat{\theta}:=\Phi^*(\rho\cdot \omega_R)$ satisfies $d\widehat{\theta}=0, \widehat{\theta}\wedge \widehat{\theta}=0$, and $\widehat{\theta}\wedge \theta=0$ (the last equation follows from Lemma \ref{lem:key_pointwise} $(ii)$), 
it follows from Lemma \ref{lem:hodge_index_thm} that we may assume $\widehat{\theta}\in{\rm SP}(\alpha)$ by replacing $\rho$ with $A\cdot\rho$ for a suitable positive constant $A$. 
Therefore it follows that $\alpha=\Phi^*\alpha_R$ for the class $\alpha_R:=\{\rho\cdot \omega_R\}$. Note that $\alpha_R$ is a K\"ahler class of $R$ since $\textstyle\int_R\rho\cdot \omega_R>0$ holds. \\
$(ii)$ Take a K\"ahler form $\omega_R$ of $R$ such that $\{\omega_R\}=\alpha_R$. 
Then, by the $\ddbar$-lemma, any element $\widehat{\theta}$ of ${\rm SP}(\alpha)$ can be written as $\widehat{\theta}=\Phi^*\omega_R+\sqrt{-1}\ddbar \vp$ by using some element $\vp\in{\rm PSH}^\infty(X, \Phi^*\omega_R)$. 
As $\Phi^*\omega_R$ is zero along each fiber, 
$\vp$ is plurisubharmonic along each fiber of $\Phi$, from which one can deduce that $\vp$ is $\Phi$-fiberwise constant by the maximum principle. Therefore $\widehat{\theta}$ is also zero along the fibers. \\
$(iii)$ Let $x\in X$ be a regular point of $\Phi$. 
Take a K\"ahler form $\omega_R$ of $R$ such that $\{\omega_R\}=\alpha_R$ and a function $f\colon R\to \mathbb{R}$ of class $C^\infty$ such that $(df)_{\Phi(x)}\not=0$. 
As $\omega_R>0$ and $X$ is compact, $\ve f\in {\rm PSH}^\infty(R, \omega_R)$ holds for a sufficiently small positive constant $\ve$. 
Thus, by letting $\theta':=\Phi^*\omega_R\in {\rm SP}(\alpha)$, one has that $\vp:=\ve\cdot (f\circ \Phi)$ is an element of ${\rm PSH}^\infty(X, \theta')$. 
As $(d\vp)_x\not=0$ by construction, it follows that $x\not\in K_\alpha$, from which the assertion holds. \\
$(iv)$ Take a point $x\in X\setminus K_\alpha$. 
It follows from Lemma \ref{lem:K_alpha} that there exists an element $\vp\in {\rm PSH}^\infty(X, \theta')$ such that $(d\vp)_x\not=0$, where $\theta'=\Phi^*\omega_R$ is the form as above. 
Note that $\vp$ is plurisubharmonic along each fiber of $\Phi$ (since $\theta'$ is fiberwise zero), from which one can deduce that $\vp$ is $\Phi$-fiberwise constant by the maximum principle. 
As we have seen in \S \ref{sec:3}, one can consider the foliation $\mathcal{F}(\theta', \widehat{\vp}, B)$ for the function $\widehat{\vp}:=\log(1+e^\vp)$ and a small neighborhood $B$ of $x$. 
Let $Y\subset B$ be the leaf of $\mathcal{F}(\theta', \widehat{\vp}, B)$ which passes through the point $x$. 
As $\vp$ is $\Phi$-fiberwise constant and $\vp|_Y$ is constant by Lemma \ref{lem:key_pointwise} $(i)$, 
it follows that $Y\subset \Phi^{-1}(p)$ for some $p\in R$ by shrinking $B$ if necessary. 
Therefore, 
it follows that $x$ is a non-singular point of the fiber $\Phi^{-1}(p)$. 
\end{proof}

As is clear by the argument in the proof of Proposition \ref{prop:case_a} $(ii)$, 
we need to define the foliation $\mathcal{F}_\alpha$ so that its leaves are the fibers of $\Phi$. 
Note that, once we adopt such definition of $\mathcal{F}_\alpha$, it easily follows that $K_\alpha^{\rm ess}=\emptyset$ by considering the $\mathcal{F}_\alpha$-adaptive function $h_R\circ \Phi$ on $X$, where, for two different regular values $p$ and $q$ of $\Phi$, $h_R$ is a harmonic function on $R\setminus \{p, q\}$ such that $h_R(w)\to -\infty$ holds as $w\to p$ and that $h_R(w)\to +\infty$ holds as $w\to q$. 

\subsection{Observation when the assertion $(b)$ holds for any $(\theta, \psi)$}\label{sec:4.2}
Next, let us investigate the case where the assertion $(b)$ of Theorem \ref{thm:main_sec3} holds for any $\theta\in {\rm SP}(\alpha)$ and any $\psi\in{\rm PSH}^\infty(X, \theta)\setminus\mathbb{R}$. 
In this case, the following {\bf Condition $(\heartsuit)$} holds: 
\begin{description}
\item[Condition $(\heartsuit)$] For any $\theta\in {\rm SP}(\alpha)$, $\psi\in{\rm PSH}^\infty(X, \theta)\setminus\mathbb{R}$, $r\in (I_\psi)_{\rm reg}$, and any connected component $A$ of the preimage $\psi^{-1}(r)$, there does not exist a non-constant $\mathcal{F}(\theta, \psi)$-leafwise constant $\mathbb{R}$-valued function on $A$ of class $C^\infty$. 
For any such $\theta$ and $\psi$, 
there exists an $\mathcal{F}(\theta, \psi)$-adaptive function $h\colon \overline{U}\to [-\infty, +\infty]$ such that $h|_U=\chi_U\circ \psi|_U$ for some strictly increasing function $\chi_U\colon \psi(U)\to \mathbb{R}$ of class $C^\infty$ for any connected component $U$ of $U_\psi$. 
\qed
\end{description}

In what follows, we use the following notation for $\theta\in {\rm SP}(\alpha)$: 
Let $\mathcal{U}(\theta)$ be the set of all the triples $(\psi, U, J)$ such that $\psi\in {\rm PSH}^\infty(X, \theta)\setminus \mathbb{R}$, $J$ is a connected open interval included in $(I_\psi)_{\rm reg}$, and that $U$ is a connected component of $\psi^{-1}(J)$. 
Let $\mathcal{U}_c(\theta)$ be the set of all the triples $(\psi, U, J)\in \mathcal{U}(\theta)$ such that $J$ is a relatively compact connected open interval included in $(I_\psi)_{\rm reg}$. 
For $(\psi, U, J)\in \mathcal{U}(\theta)$, 
we denote by $\mathcal{H}(\psi, U, J)$ the set of all the pluriharmonic functions $h$ on $U$ such that $h=\chi\circ \psi|_U$ holds for some function $\chi\colon J\to \mathbb{R}$. 
We also use the notation $\mathcal{H}^*(\psi, U, J):=\mathcal{H}(\psi, U, J)\setminus \mathbb{R}$ for each $(\psi, U, J)\in \mathcal{U}(\theta)$. 
Note that the set $\mathcal{H}^*(\psi, U, J)$ is non-empty by {\bf Condition $(\heartsuit)$}. 
First we show the following: 
\begin{lemma}\label{lem:S42first}
Let $(\psi, U, J)$ be an element of $\mathcal{U}(\theta)$ 
and $h=\chi_U\circ \psi|_U$ be (the restriction of) a pluriharmonic function as in {\bf Condition $(\heartsuit)$}. 
Then it holds that 
$\mathcal{H}(\psi, U, J)=\{c_1h+c_2\mid c_1, c_2\in \mathbb{R}\}$ 
and 
$\mathcal{H}^*(\psi, U, J)=\{c_1h+c_2\mid c_1, c_2\in \mathbb{R}, c_1\not=0\}$. 
Especially, any element of $\mathcal{H}^*(\psi, U, J)$ is the restriction of an $\mathcal{F}(\theta, \psi)$-adaptive function as in {\bf Condition $(\heartsuit)$}. 
\end{lemma}

\begin{proof}
Take a pluriharmonic functions $\widetilde{h}$ on $U$ such that $\widetilde{h}=\widetilde{\chi}\circ \psi|_U$ holds for some function $\widetilde{\chi}\colon J\to \mathbb{R}$, 
and (a restriction of) an $\mathcal{F}(\theta, \psi)$-adaptive function $h|_U=\chi_U\circ \psi|_U$ as in {\bf Condition $(\heartsuit)$}. 
As $\chi_U$ is strictly increasing, $\chi_U$ is bijective to the image. 
Therefore one has that $\widetilde{h}=\widetilde{\chi}\circ \chi_U^{-1}\circ h$ holds on $U$. 
As $\sqrt{-1}\ddbar h\equiv 0$ and $\sqrt{-1}\ddbar\widetilde{h}\equiv 0$, it holds that 
$((\widehat{\chi}\circ \chi_U^{-1})''\circ h)\cdot \sqrt{-1}\del h\wedge \delbar h\equiv 0$. 
Thus one has that $(\widetilde{\chi}\circ \chi_U^{-1})''\equiv 0$, from which it follows that $\widetilde{\chi}\circ \chi_U^{-1}$ is a polynomial of degree at most one. 
\end{proof}

\subsubsection{Comparison with another element of ${\rm PSH}^\infty(X, \theta)$}
Let $\theta$ be an element of ${\rm SP}(\alpha)$. 
In this subsection, we consider two elements $\vp$ and $\psi$ of ${\rm PSH}^\infty(X, \theta)$. 
\begin{lemma}\label{lem:comparison}
Assume that {\bf Condition $(\heartsuit)$} holds and that $U_\vp\cap U_\psi\not=\emptyset$. 
Let $U$ be a connected component of $U_\vp\cap U_\psi$ and $J$ be the image $\psi(U)$. 
Then the following holds: \\
$(i)$ $(\psi, U, J)$ is an element of $\mathcal{U}(\theta)$. \\
$(ii)$ $\vp|_U=\chi\circ \psi|_U$ holds for some strictly increasing or strictly decreasing function $\chi\colon J\to \mathbb{R}$. \\
$(iii)$ $(\vp, U, \chi(J))$ is an element of $\mathcal{U}(\theta)$. \\
$(iv)$ The foliation $\mathcal{F}(\theta, \psi)|_U$ coincides with $\mathcal{F}(\theta, \vp)|_U$. 
\end{lemma}

\begin{proof}
Let $W$ be the connected component of $\psi^{-1}(J)$ which includes $U$. 
As $W\subset U_\psi$, it follows from Lemma \ref{lem:key_local} $(i)$ that $\vp$ is $\mathcal{F}(\theta, \psi)|_W$-leafwise constant. 
By {\bf Condition $(\heartsuit)$}, it holds that $\vp|_{W\cap\psi^{-1}(r)}$ is constant for each $r\in J$, from which it follows that there exists a function $\chi\colon J\to \mathbb{R}$ such that $\vp|_W=\chi\circ \psi|_W$ holds. 

First, let us show that $\chi'(r)\not=0$ for each $r\in J$. 
Take $r\in J$. As $J$ is the image of $U$ by $\psi$, there exists a point $x\in U$ 
such that $\psi(x)=r$. As $x\in U_\vp\cap U_\psi$, one has that $(d\vp)_x\not=0$ and $(d\psi)_x\not=0$. 
Thus it follows from the equation $(d\vp)_x=\chi'(r)\cdot (d\psi)_x$ that $\chi'(r)\not=0$. 
Therefore one has that $\chi$ is strictly increasing or strictly decreasing. 

Again by the equation $d\vp=(\chi'\circ \psi)\cdot d\psi$ on $W$, it follows that $W\subset U_\vp$, from which one has that $U=W$. 
The assertions $(i)$, $(ii)$, and $(iii)$ follows from these observation. 
The assertion $(iv)$ follows from the fact that $T_{\mathcal{F}(\theta, \psi)}=(\del \psi)^\perp$ and $T_{\mathcal{F}(\theta, \vp)}=(\del \vp)^\perp$ holds on $U$ (by Lemma \ref{lem:key_pointwise} $(i)$) and the equation $d\vp=(\chi'\circ \psi)\cdot d\psi$ on $U$. 
\end{proof}

\subsubsection{Observation on foliation adaptive functions}
Let $\theta$ be an element of ${\rm SP}(\alpha)$. 
In this subsection, we investigate foliation adaptive functions under the assumption that {\bf Condition $(\heartsuit)$} holds. 
As a preliminary, first we show the following:
\begin{lemma}\label{lem:4.3}
When {\bf Condition $(\heartsuit)$} holds, the following holds for $(\psi, U, I)\in \mathcal{U}(\theta)$, $h\in \mathcal{H}^*(\psi, U, I)$, and $J:=h(U)$: \\
$(i)$ $\theta|_U=(\rho\circ h)\cdot \sqrt{-1}\del h\wedge \delbar h$ holds for some function $\rho\colon J\to \mathbb{R}_{\geq 0}$. \\
$(ii)$ For any $\vp\in {\rm PSH}^\infty(X, \theta)$, there exists a function $\chi\colon J\to \mathbb{R}$ such that $\chi''\geq -\rho$ and $\vp|_U=\chi\circ h$ hold, where $\rho$ is the function as in the assertion $(i)$. \\
$(iii)$ For any $\vp\in {\rm PSH}^\infty(X, \theta)$ with $U\subset U_\vp$, $(\vp, U, \vp(U))\in\mathcal{U}(\theta)$ and $\mathcal{H}(\vp, U, \vp(U))=\mathcal{H}(\psi, U, I)$ hold. 
\end{lemma}

\begin{proof}
The assertion $(i)$ follows from Lemma \ref{lem:key_local} $(ii)$, {\bf Condition $(\heartsuit)$}, and Lemma \ref{lem:S42first}. 
The assertion $(ii)$ follows from Lemma \ref{lem:key_local} $(i)$, {\bf Condition $(\heartsuit)$}, and Lemma \ref{lem:S42first}. 
Let us show the assertion $(iii)$. 
From Lemma \ref{lem:comparison} $(iii)$, it follows that $(\vp, U, \vp(U))\in\mathcal{U}(\theta)$. 
Take an element $\widehat{h}\in \mathcal{H}(\vp, U, \vp(U))$. 
By Lemma \ref{lem:S42first}, there exists a function $F\colon \vp(U)\to \mathbb{R}$ such that $\widehat{h}=F\circ \vp|_U$. 
Thus one has $\widehat{h}=(F\circ \chi)\circ h$, where $\chi$ is the function as in the assertion $(ii)$. As both $\widehat{h}$ and $h$ is pluriharmonic and 
$\sqrt{-1}\ddbar (F\circ \chi)\circ h=((F\circ \chi)''\circ h)\cdot \sqrt{-1}\del h\wedge \delbar h$, one has that $F\circ \chi$ is a polynomial of at most degree one, from which the assertion follows. 
\end{proof}

In what follows, we often consider the following configuration: 
\begin{description}
\item[Configuration $(\natural)$] A domain $W$ of a connected compact K\"ahler manifold $X$ satisfies the following conditions: The boundary $\del W$ consists of two connected components $H_+$ and $H_-$, and there exist open neighborhoods $W_\pm$ of $H_\pm$ in $X$ and $\vp_\pm\in{\rm PSH}^\infty(X, \theta)$ for some $\theta\in {\rm SP}(\alpha)$ such that $U_\pm:=W_\pm\cap W$ and $(a_\pm, b_\pm):=\vp_\pm(U_\pm)$ satisfies 
$\overline{U_+}\cap \overline{U_-}=\emptyset$, 
$(\vp_\pm, U_\pm, (a_\pm, b_\pm))\in \mathcal{U}_c(\theta)$, 
$H_-=\vp_-^{-1}(a_-)\cap \overline{U_-}$, and 
$H_+=\vp_+^{-1}(b_+)\cap \overline{U_+}$. \qed
\end{description}

For such $X, W, H_\pm$, and $\vp_\pm\colon U_\pm\to (a_\pm, b_\pm)$ as in {\bf Configuration $(\natural)$}, we have the following: 
\begin{lemma}\label{lem:natural}
Assume that {\bf Condition $(\heartsuit)$} holds. 
Let $X, W, H_\pm$, and $\vp_\pm\colon U_\pm\to (a_\pm, b_\pm)$ be as in {\bf Configuration $(\natural)$}. 
Then either $(i)$ or $(ii)$ holds for $V:=X\setminus \overline{W}$. \\
$(i)$ $V$ is connected and there exists a continuous function $h_V\colon \overline{U_+\cup V\cup U_-}\to \mathbb{R}$ which is a non-constant pluriharmonic function on the interior of $\overline{U_+\cup V\cup U_-}$ such that $h_V|_{U_\pm}\in \mathcal{H}^*(\vp_\pm, U_\pm, (a_\pm, b_\pm))$. \\
$(ii)$ $V$ consists of two connected components $V^+$ and $V^-$ such that $\del V^-=H_-$ and $\del V^+=H_+$, and there exists a continuous function $h_W\colon \overline{W}\to \mathbb{R}$ such that $h_W|_W$ is a non-constant pluriharmonic function and that $h_W|_{U_\pm}\in \mathcal{H}^*(\vp_\pm, U_\pm, (a_\pm, b_\pm))$. 
\end{lemma}

\begin{proof}
Denote by $\widetilde{V}$ the interior of $\overline{U_+\cup V\cup U_-}$. 
First let us note that it follows by considering 
Ehresmann's lemma for $\vp_\pm|_{U_\pm}$ and 
Mayer--Vietoris sequence for the covering $\{\widetilde{V}, W\}$ of $X$ that either of the following holds: 
$V$ is connected, or $V$ consists of two connected components $V^+$ and $V^-$ such that $\del V^-=H_-$ and $\del V^+=H_+$. 

Let $\widetilde{U_\pm}$ be the connected component of $\vp_\pm^{-1}((I_{\vp_\pm})_{\rm reg})$ which includes $U_\pm$. 
Denote by $(c_\pm, d_\pm)$ the interval $\vp_\pm(\widetilde{U_\pm})$. 
Note that $c_\pm<a_\pm<b_\pm<d_\pm$ holds, since $(\vp_\pm, U_\pm, (a_\pm, b_\pm))\in \mathcal{U}_c(\theta)$. 
Take an element $h_\pm\in \mathcal{H}^*(\vp_\pm, \widetilde{U}_\pm, (c_\pm, d_\pm))$ as in {\bf Condition $(\heartsuit)$}. 
Set $a_\pm':=h_\pm(\vp_\pm^{-1}(a_\pm)\cap \widetilde{U}_\pm)$,  $b_\pm':=h_\pm(\vp_\pm^{-1}(b_\pm)\cap \widetilde{U}_\pm)$, 
$c_\pm':=\inf h_\pm$, and 
$d_\pm':=\sup h_\pm$. 
Then, for a function $\rho_\pm\colon (c_\pm', d_\pm')\to \mathbb{R}_{\geq 0}$ of class $C^\infty$ such that ${\rm Supp}\,\rho_\pm=[a_\pm', b_\pm']$, 
it follows from Lemma \ref{lem:4.5} below that 
\[
\theta_\pm := \begin{cases}
(\rho_\pm\circ h_\pm)\cdot \sqrt{-1}\del h_\pm\wedge \delbar h_\pm & \text{on}\ \widetilde{U_\pm}\\
0 &  \text{on}\ X\setminus \widetilde{U_\pm}
\end{cases}
\]
is an element of ${\rm SP}(\alpha)$ by replacing $\rho_\pm$ with $A_\pm\cdot \rho_\pm$ for some $A_\pm>0$ if necessary. 
Take a function $f\colon X\to \mathbb{R}$ such that $\theta_+=\theta_-+\sqrt{-1}\ddbar f$, whose existence is assured by the $\ddbar$-lemma. 
Note that $\mp f\in {\rm PSH}^\infty(X, \theta_\pm)$. 
It follows from Lemma \ref{lem:4.3} $(ii)$ that $f|_{\widetilde{U_\pm}}=\chi_\pm\circ h_\pm$ for some function $\chi_\pm\colon (c_\pm', d_\pm')\to \mathbb{R}$. 
As $\theta_\pm|_{U_\mp}\equiv 0$ by construction, 
one can deduce from the equation $\sqrt{-1}\ddbar f|_{\widetilde{U_\pm}}=\sqrt{-1}\ddbar (\chi_\pm\circ h_\pm)$ that $\chi_\pm''=\pm\rho_\pm$ holds on a neighborhood of $[a_\pm', b_\pm']$. 

First, let us consider the case where $\chi_-|_{(a_-'-\ve, a_-')}$ is not a constant function for any small positive constant $\ve$. 
In this case, consider the function $h_V:=f|_V$. 
As $\theta_+=\theta_-+\sqrt{-1}\ddbar f$ and $V\cap {\rm Supp}\,\theta_\pm=\emptyset$, $h_V$ is a pluriharmonic function on $V$. 
By the assumption on $\chi_-|_{(a_-'-\ve, a_-')}$, $h_V$ is non-constant. 
Therefore $V$ is connected, since otherwise the contradiction occurs by considering the maximum principle for $h_V|_{V^-}$ (Note that, as $f|_{\widetilde{U_-}}=\chi_-\circ h_-$, $f|_{\del V^-}=f|_{H_-}$ is constant). 
As $\chi_\pm''\equiv 0$ on the complement of $[a_\pm', b_\pm']$, one can easily extend $h_V$ by using elements of $\mathcal{H}(\vp_\pm, U_\pm, (a_\pm, b_\pm))$ to construct a non-constant pluriharmonic function on $\widetilde{V}$ (by Lemma \ref{lem:S42first}), from which one has that the assertion $(i)$ holds in this case. 

Next, let us consider the case where $\chi_-|_{(a_-'-\ve, a_-')}\equiv A$ holds for a small positive constant $\ve$ and a constant $A\in \mathbb{R}$. 
Note that $\chi_-$ is a strictly decreasing function on $(b_-', b_-'+\ve)$ in this case, since $\chi_-''=-\rho_-$ on $[a_-', b_-']$ and $\chi_-'|_{(a_-'-\ve, a_-')}\equiv 0$. 
Let us consider the function $h_W:=f|_{W_0}$, where $W_0:=W\setminus \overline{U_+\cup U_-}$. 
As $\theta_+=\theta_-+\sqrt{-1}\ddbar f$ and $W_0\cap {\rm Supp}\,\theta_\pm=\emptyset$, $h_W$ is a pluriharmonic function on $W_0$. 
The non-constantness of $\chi_-|_{(b_-', b_-'+\ve)}$ implies that $h_W$ is non-constant. As $\chi_\pm''\equiv 0$ on the complement of $[a_\pm', b_\pm']$, one can easily extend $h_W$ by using some elements of $\mathcal{H}(\vp_\pm, U_\pm, (a_\pm, b_\pm))$ to construct a non-constant pluriharmonic function on $W$. 
Finally, let us show that $V$ is not connected. 
Assume that $V$ is connected. 
Then, as $f|_V$ is pluriharmonic and $f\equiv A$ on $\{x\in \widetilde{U}_-\mid a_-'-\ve < h_-(x) < a_-'\} (\subset V)$, it follows from the identity theorem that $f|_V\equiv A$. Thus $\chi'\equiv 0$ on $(b_+', b_+'+\delta)$ for a small positive number $\delta$, from which it follows that $\chi_+$ is a strictly decreasing function on $(a_+'-\delta, a_+')$ for small $\delta>0$, since $\chi_+''=\rho_+$ on $[a_+', b_+']$ and $\chi_+'|_{(b_+', b_+'+\delta)}\equiv 0$. 
Therefore, from these observation on $\chi_\pm$ and the maximum principle for the pluriharmonic function $f|_{W_0}$, one has that the maximum vale of $f|_{\overline{W}}$ is attained along $H_-$ and the minimum value of $f|_{\overline{W}}$ is attained along $H_+$, which contradicts to the maximum principle for the non-constant plurihamornic function $f|_{W_0}$ since $f|_{H_\pm}\equiv A$. 
\end{proof}

\begin{lemma}\label{lem:4.5}
Let 
$\theta$ be an element of ${\rm SP}(\alpha)$, 
$(\vp, U, (c,d))$ an element of $\mathcal{U}(\theta)$, 
and $h$ be an element of $\mathcal{H}^*(\vp, U, (c, d))$. 
For a function $\rho\colon (c, d)\to \mathbb{R}_{\geq 0}$ of class $C^\infty$ whose support is relatively compact in $(c, d)$, there exists a positive constant $A$ such that $A\cdot \theta_\rho\in {\rm SP}(\alpha)$, where 
\[
\theta_\rho :=\begin{cases}
(\rho\circ h)\cdot \sqrt{-1}\del h\wedge \delbar h & \text{on}\ U\\
0 & \text{on}\ X\setminus U
\end{cases}
\]
\end{lemma}

\begin{proof}
Note that $d\theta_\rho\equiv 0$ and $\theta_\rho\wedge \theta_\rho\equiv 0$ by definition. 
As it follows from Lemma \ref{lem:4.3} $(i)$ that $\theta\wedge \theta_\rho\equiv 0$, the lemma follows from Lemma \ref{lem:hodge_index_thm} by letting 
\[
A:=\frac{\int_X\theta\wedge \omega^{{\rm dim}\,X-1}}{\int_X\theta_\rho\wedge \omega^{{\rm dim}\,X-1}}
\]
for a suitable K\"ahler form $\omega$ of $X$. 
\end{proof}

Moreover, we have the following: 
\begin{lemma}\label{lem:natural2}
Assume that {\bf Condition $(\heartsuit)$} holds. 
Let $X, W, H_\pm$, and $\vp_\pm\colon U_\pm\to (a_\pm, b_\pm)$ be as in {\bf Configuration $(\natural)$}. \\
$(i)$ Assume that the assertion $(i)$ of Lemma \ref{lem:natural} holds. 
Let $\mathcal{F}_V$ be the foliation on $V:=X\setminus \overline{W}$ which is defined by $T_{\mathcal{F}_V}=(\del h_V|_V)^\perp$. 
Then $h_V$ is $\mathcal{F}_V$-adaptive and $K_\alpha\cap V=\{x\in V\mid (dh_V)_x=0\}$. 
In this case, any element of ${\rm SP}(\alpha)$ is zero along each leaf of $\mathcal{F}_V|_{V\setminus K_\alpha}$. \\
$(ii)$ Assume that the assertion $(ii)$ of Lemma \ref{lem:natural} holds. 
Let $\mathcal{F}_W$ be the foliation on $W$ which is defined by $T_{\mathcal{F}_W}=(\del h_W|_W)^\perp$. 
Then $h_W$ is $\mathcal{F}_W$-adaptive and $K_\alpha\cap W=\{x\in W\mid (dh_W)_x=0\}$. 
In this case, any element of ${\rm SP}(\alpha)$ is zero along each leaf of $\mathcal{F}_W|_{W\setminus K_\alpha}$. 
\end{lemma}

\begin{proof}
Here we only show the assertion $(ii)$, since $(i)$ is shown by the same argument. 
In what follows, we assume that the assertion $(ii)$ of Lemma \ref{lem:natural} holds and use the notation in the proof of Lemma \ref{lem:natural}. 
Note that $\mathcal{F}_W|_{W\cap \widetilde{U}_\pm}$ coincides with the foliation $\mathcal{F}(\theta, \vp_\pm)|_{W\cap \widetilde{U}_\pm}$, since $(\del h_\pm)^\perp=(\del \vp_\pm)^\perp$ holds on $T_{\widetilde{U}_\pm}$ by Lemma \ref{lem:S42first}. 
Note also that $h_W$ is clearly $\mathcal{F}_W$-adaptive by definition of $\mathcal{F}_W$. 

Take an element $\widehat{\theta}\in {\rm SP}(\alpha)$. 
The $\mathcal{F}_W$-leafwise triviality of $\widehat{\theta}$ on $W\cap\widetilde{U}_\pm$ follows from Proposition \ref{prop:key} $(ii)$. 
On a neighborhood $B_0$ of each point $x_0$ of $\{x\in W_0\mid (dh_W)_x\not=0\}=\{x\in W_0\mid (df)_x\not=0\}$, it follows from Lemma \ref{lem:key_pointwise} $(i)$ that $\mathcal{F}_W|_{B_0}=\mathcal{F}(\theta_-, \log(1+e^f), B_0)$ (Recall that $h_W|_{W_0}=f|_{W_0}$ and $f\in {\rm PSH}^\infty(X, \theta_-)$). Thus one has that $\widehat{\theta}$ is $\mathcal{F}_W|_{B_0}$-leafwise trivial by Lemma \ref{lem:key_pointwise}, \ref{lem:key_local} $(i)$ and the $\ddbar$-lemma. 

Therefore, it is sufficient to show that $W_0\cap K_\alpha=\{x\in W_0\mid (dh_W)_x=0\}$. 
As the inclusion $W_0\cap K_\alpha\subset \{x\in W_0\mid (dh_W)_x=0\}=\{x\in W_0\mid (df)_x=0\}$ simply follows from the definition of $K_\alpha$, we will show the opposite inclusion in what follows. 

Take a point $x_0\in W_0$ such that $(dh_W)_{x_0}=0$. 
Let $B$ be a small open neighborhood of $x_0$. 
Let $F$ be a holomorphic function on $B$ such that ${\rm Re}\,F=h_W|_B$ holds. 
By adding constants to $h_W$ and $F$ if necessary, we may assume that $h_W(x_0)=F(x_0)=0$. 
We also assume that $B$ is small enough so that any connected component of $F^{-1}(0)$ and any connected component of $S:=\{x\in B\mid (dF)_x=0\}$ contain the point $x_0$. 
Note that $S\subset F^{-1}(0)$ (see Lemma \ref{lem:fol_sing_lwconst}). 

We show $x_0\in K_\alpha$ by contradiction, by assuming that there exists $\vp\in {\rm PSH}^\infty(X, \theta_-)$ such that $(d\vp)_{x_0}\not=0$ (Here we used Lemma \ref{lem:K_alpha}). 
By shrinking $B$ if necessary, we may assume that $\vp$ has no critical point in $B$. 
By applying Lemma \ref{lem:key_local} $(i)$ for $(\theta_-, \log(1+e^\vp))$ and $f\in {\rm PSH}^\infty(X, \theta_-)$, 
one has that there exists a function $G\colon B\to \mathbb{C}$ such that $\del h_W = G\cdot \del \vp$ and that $|G|$ is $\mathcal{F}(\theta_-, \log(1+e^\vp)), B)$-leafwise constant. 
As $S=\{x\in B\mid |G(x)|=0\}$, $S$ is the union of some leaves of $\mathcal{F}(\theta_-, \log(1+e^\vp)), B)$. 
As $S$ is a connected analytic subset of $B$ and each leaf of $\mathcal{F}(\theta_-, \log(1+e^\vp)), B)$ is a complex submanifold of $B$ of codimension $1$, it follows that $x_0$ is a regular point of $S$ and ${\rm dim}(S, x_0)={\rm dim}\,X-1$. 
By applying  Lemma \ref{lem:key_local} $(i)$ for $(\theta_-, \log(1+e^f))$ and $\vp\in {\rm PSH}^\infty(X, \theta_-)$, 
one has that $\vp$ is $\mathcal{F}_W|_{B\setminus S}$-leafwise constant, 
since $\mathcal{F}_W|_{B\setminus S}=\mathcal{F}(\theta_-, \log(1+e^f), B)$ by Lemma \ref{lem:key_pointwise} $(i)$. 
Thus, by Lemma \ref{lem:fol_sing_lwconst}, there exist a coordinate system $w=(w^1, w^2, \dots, w^n)$ of $B$ with $x_0=(0, 0, \dots, 0)$ and an integer $m\in \mathbb{Z}_{>1}$ such that $F(w)=(w^1)^m$ and that $\vp(w)=\chi(w^1)$ holds on $B$ for some function $\chi\colon w^1(B)\to \mathbb{R}$. 
Note that $\chi$ is of class $C^\infty$, since $\vp$ is $C^\infty$. 

In what follows, we show the existence of a function $\overline{\chi}\colon (-\delta, \delta)\to \mathbb{R}$ of class $C^\infty$ for a positive number $\delta$ such that $\chi(w^1) = \overline{\chi}({\rm Re}\,(w^1)^m)$. 
Note that, if such a function exists, then it holds that $\vp = \overline{\chi}\circ h_W$, which proves the lemma since the calculation $(d\vp)_{x_0}=\overline{\chi}'(0)\cdot (dh_W)_{x_0}=0$ contradicts to the assumption $(d\vp)_{x_0}\not=0$. 

As we have seen in \S \ref{sec:fund_obs_F-adaptive}, any element of $(-\delta, 0)\cup (0, \delta)$ is a regular value of $h_W$ for some positive number $\delta$. 
Set $\ve:=\sqrt[m]{\delta}$. 
By shrinking $B$, we may assume that $B=\{(w^1, w^2, \dots, w^n)\mid |w^j|<\ve\ \text{for all}\ j=1, 2, \dots, n\}$. 
For $r\in (I_f)_{\rm reg}\cap (-\delta, \delta)$, denote by $A_r$ the connected component of $f^{-1}(r)$ which intersects $B$. 
By {\bf Condition $(\heartsuit)$}, $\vp|_{A_r}$ is constant for any $r\in (I_f)_{\rm reg}\cap (-\delta, \delta)$. 
As $A_r\cap B=\{w\in B\mid {\rm Re}\,(w^1)^m=r\}$, 
it follows that $\chi(w^1) = \chi(\zeta_m\cdot w^1)$ holds on a dense subset of $\Delta_\ve:=\{w^1\in\mathbb{C}\mid |w^1|<\ve\}$, where $\zeta_m:=\exp(2\pi\sqrt{-1}/m)$. 
As $\chi$ is continuous, one has that $\chi$ is invariant under the rotation by $\zeta_m$. 
Therefore there exists a continuous function $\eta\colon \Delta_\delta\to \mathbb{R}$ such that $\chi(w^1)=\eta((w^1)^m)$, where $\Delta_\delta:=\{\xi\in\mathbb{C}\mid |\xi|<\delta\}$. 
Again by the same argument, one can show that there exists a continuous function $\overline{\chi}\colon (-\delta, \delta)\to \mathbb{R}$ such that $\eta(\xi)=\overline{\chi}({\rm Re}\,\xi)$. 
By construction, one has $\chi(w^1) = \overline{\chi}({\rm Re}\,(w^1)^m)$. 
As $\chi$ is of class $C^\infty$ and the map $\Delta_\ve\setminus\{0\}\ni w^1\mapsto {\rm Re}\,(w^1)^m\in (-\delta, \delta)$ is surjective submersion of class $C^\infty$, $\overline{\chi}$ is also smooth. 
\end{proof}


\section{Proof of main theorems} \label{sec:5}
Let $X$ be a connected compact K\"ahler manifold. 
For a class $\alpha\in H^{1, 1}(X, \mathbb{R})$ with ${\rm nd}(\alpha)=1$ and $\#{\rm SP}(\alpha)>1$, we show Theorems \ref{thm:main1} and \ref{thm:main2}. 

\subsection{Outline of the proof}\label{sec:5.1}
As we have seen in \S \ref{sec:4.1}, 
the assertions in Case I of Theorem \ref{thm:main2} holds when the assertion $(a)$ of Theorem \ref{thm:main_sec3} holds for some $\theta\in {\rm SP}(\alpha)$ and $\psi\in{\rm PSH}^\infty(X, \theta)\setminus\mathbb{R}$. Theorem \ref{thm:main1} is clear in this case. 
Therefore it is sufficient to show the theorems by assuming {\bf Condition $(\heartsuit)$}. 

In what follows, we assume {\bf Condition $(\heartsuit)$} and use the notation in \S \ref{sec:4.2}. 
By Lemma \ref{lem:natural}, either of the following two conditions holds: 
\begin{description}
\item[Condition $(\clubsuit)$] There exist $\theta\in {\rm SP}(\alpha)$ and $(\psi, U, (a, b))\in \mathcal{U}_c(\theta)$ such that $X\setminus \overline{U}$ is connected. \qed
\item[Condition $(\diamondsuit)$] For any $\theta\in {\rm SP}(\alpha)$ and any $(\psi, U, (a, b))\in \mathcal{U}_c(\theta)$, $X\setminus \overline{U}$ consists of two connected components $V^+$ and $V^-$ such that $\del V^-=\psi^{-1}(a)\cap\overline{U}$, $\del V^+=\psi^{-1}(b)\cap\overline{U}$. \qed
\end{description}

First let us consider the case where {\bf Condition $(\clubsuit)$} holds. 
Fix $\theta\in {\rm SP}(\alpha)$ and $(\psi, U, (a, b))\in \mathcal{U}_c(\theta)$ such that $V:=X\setminus \overline{U}$ is connected. 
Take an element $(\psi, \widetilde{U}, (c, d))\in \mathcal{U}(\theta)$ such that $U\subset \widetilde{U}$ and $c<a<b<d$, and an element $h_U\in \mathcal{H}^*(\psi, \mathcal{U}, (c, d))$. 
Note that $\{\widetilde{U}, V\}$ is a covering of $X$ by two domains. 
Define a holomorphic foliation $\mathcal{F}_\alpha$ on $X$ by letting 
\[
T_{\mathcal{F}_\alpha} = \begin{cases}
(\del h_U)^\perp & \text{on}\ \widetilde{U}\\
(\del h_V)^\perp & \text{on}\ V
\end{cases}, 
\]
where $h_V$ is the function as in Lemma \ref{lem:natural} $(i)$. 
Then it follows from 
Lemma \ref{lem:S42first} and Lemma \ref{lem:natural2} $(i)$ that the assertions in Case I\!I of Theorem \ref{thm:main2} holds in this case (by Proposition \ref{prop:key} $(ii)$ and Lemma \ref{lem:natural2} $(i)$. Note that $K_\alpha\cap \widetilde{U}=\emptyset$ since $\widetilde{U}\subset \psi^{-1}((I_\psi)_{\rm reg})$). 

Therefore, it is sufficient to show that the assertions in Case I\!I\!I of Theorem \ref{thm:main2} holds in the case where {\bf Condition $(\diamondsuit)$} holds. 
Fix an element $\theta_0\in {\rm SP}(\alpha)$, 
$(\psi_0, W_0, (c_0, d_0))\in \mathcal{U}(\theta_0)$ such that $W_0$ is a connected component of $\psi_0^{-1}((I_{\psi_0})_{\rm reg})$, 
and $h_0\in \mathcal{H}^*(\psi_0, W_0, (c_0, d_0))$. 
By Lemma \ref{lem:S42first}, we may assume that $h_0=\chi_0\circ \psi_0$ for a strictly increasing function $\chi_0\colon (c_0, d_0)\to \mathbb{R}$ by replacing $h_0$ with $-h_0$ if necessary. 
Fix also real numbers $a_0$ and $b_0$ such that $c_0<a_0<b_0< d_0$. 
By {\bf Condition $(\diamondsuit)$}, the complement $X\setminus\{x\in W_0\mid a_0\leq \psi_0(x)\leq b_0\}$ consists of two connected components, say $V_0^\pm$, such that $\del V_0^-=\{x\in W_0\mid \psi_0(x)= a_0\}$ and $\del V_0^+=\{x\in W_0\mid \psi_0(x)= b_0\}$ hold. 
In what follows, we denote by $W_0^\pm$ the domain $V_0^\pm\cup W_0$. 

In the following three subsections, 
we will construct such a function $h_\alpha$ as in Theorem \ref{thm:main2} by considering the analytic continuation of $h_0$ along the following three steps by assuming that {\bf Condition $(\heartsuit)$} and {\bf Condition $(\diamondsuit)$} hold. 
\begin{description}
\item[Step 1] For each element $(\psi, U, I)\in \mathcal{U}(\theta_0)$, we construct a function $h_U\colon W_0^-\to \mathbb{R}\cup\{-\infty\}$ such that $h_U|_{W_0}=h_0$ and that $h_U$ is pluriharmonic on a domain of $X$ which contains both $W_0$ and $U$. The construction of $h_U$ depends only on $U$ and is independent on $\psi$. 
\item[Step 2] Define a function $h_-\colon W_0^-\to \mathbb{R}\cup\{-\infty\}$ by 
$h_-(x):=\inf\{h_U(x)\mid (\psi, U, I)\in \mathcal{U}(\theta_0)\}$, and show that $h_-$ is pluriharmonic on $W_0^-\setminus M^-$, where $M^-=\{x\in W_0^-\mid h_-(x)=\textstyle\min_{W_0^-}h_-\}$. 
\item[Step 3] Construct $h^+\colon W_0^+\to \mathbb{R}\cup\{+\infty\}$ and $M^+\subset W_0^+$ in the same manner, and define $h_\alpha$ by patching $(W_0^\pm, h_\pm)$. Show that $M^+\cup M^-=K_\alpha^{\rm ess}$. 
\end{description}

We often use the following topological lemma: 
\begin{lemma}\label{lem:topological}
Assume that {\bf Condition $(\diamondsuit)$} holds. 
Let $U_1, U_2, \dots, U_N$ be relatively compact domains of $W_0^-$ such that 
$(\psi_j, U_j, \psi_j(U_j))\in \mathcal{U}_c(\theta_j)$ holds for some $\theta_j\in {\rm SP}(\alpha)$ and $\psi_j\in {\rm PSH}^\infty(X, \theta_j)$ for $j=1, 2, \dots, N$. 
Assume that $U_j\cap U_k=\emptyset$ if $j\not=k$. 
Then $W_0^-\setminus (U_1\cup U_2\cup\cdots \cup U_N)$ consists of $N+1$ connected components $L_0, L_1, \dots, L_N$. Moreover, by changing the indexes if necessary and letting $H_j^\pm$ the connected components of $\del U_j$, it holds that 
$\del L_0=\del W_0^-\cup H_1^+$, 
$\del L_1=H_1^-\cup H_2^+$, 
$\del L_2=H_2^-\cup H_3^+, \cdots, \del L_{N-1}=H_{N-1}^-\cup H_N^+$, 
and $\del L_N=H_N^-$. 
\end{lemma}

\begin{proof}
Denote by $L$ the complement $W_0^-\setminus (U_1\cup U_2\cup\cdots \cup U_N)$. 
By considering the covering of $X$ by $U_j$'s and a suitable neighborhood of $L$, 
it follows from Mayer--Vietoris sequence that the rank of $H_0(L, \mathbb{Z})$ is at most $N+1$. 
Then the assertion can be easily shown from {\bf Condition $(\diamondsuit)$} by the induction on $N$. 
\end{proof}

\subsection{Step 1}
Let $(\psi, U, I)$ be an element of $\mathcal{U}(\theta_0)$. 
In this subsection, we construct a function $h_U\colon W_0^-\to \mathbb{R}\cup\{-\infty\}$ such that $h_U|_{W_0}=h_0$ and that $h_U$ is pluriharmonic on a domain of $X$ which contains both $W_0$ and $U$. 

\subsubsection{The construction of $h_U$ when $U\cap (V_0^-\setminus \overline{W_0})=\emptyset$}
Assume that $U\cap (V_0^-\setminus \overline{W_0})=\emptyset$. 
In this case, we denote by $U^c_{\rm outside}$ the closed set $W_0^-\setminus W_0$ 
and define $h_U\colon W_0^-\to \mathbb{R}\cup\{-\infty\}$ by 
\[
h_U:=\begin{cases}
h_0 & \text{on}\ W_0\\
\inf_{W_0}h_0 & \text{on}\ U^c_{\rm outside}
\end{cases}. 
\]
Note that this construction of $h_U$ depends only on $U$ and is independent on $\psi$. 

\subsubsection{The construction of $h_U$ when $U\cap (V_0^-\setminus \overline{W_0})\not=\emptyset$ and $U\cap W_0\not=\emptyset$}
Assume that $U\cap (V_0^-\setminus \overline{W_0})\not=\emptyset$ and $U\cap W_0\not=\emptyset$. 
By Lemma \ref{lem:comparison}, the connectedness of $U$, and {\bf Condition $(\diamondsuit)$}, 
one has that $U\cap W_0=\{x\in W_0\mid c_0<\psi_0(x)<\ell\}$ for some $\ell\in (c_0, d_0]$. 

Take an element $h_1\in \mathcal{H}^*(\psi, U, I)$. 
By applying Lemma \ref{lem:comparison} $(ii)$ to $U\cap W_0$ and Lemma \ref{lem:S42first} to $h_0|_{U\cap W_0}$ and $h_1|_{U\cap W_0}$, one has that we may assume $h_0=h_1$ holds on $U\cap W_0$ by replacing $h_1$ with $c_1h_1+c_2$ for a suitable constants $c_1, c_2\in \mathbb{R}$. 
In this case, we denote by $U^c_{\rm outside}$ the closed set $W_0^-\setminus (U\cup W_0)$ 
and define $h_U\colon W_0^-\to \mathbb{R}\cup\{-\infty\}$ by 
\[
h_U:=\begin{cases}
h_0 & \text{on}\ W_0\\
h_1 & \text{on}\ U\\
\inf_{U}h_1 & \text{on}\ U^c_{\rm outside}
\end{cases}. 
\]
Note that, by Lemma \ref{lem:S42first}, this construction of $h_U$ depends only on $U$ and is independent on $\psi$. 

\subsubsection{The construction of $h_U$ when $U\cap (V_0^-\setminus \overline{W_0})\not=\emptyset$ and $U\cap W_0=\emptyset$}
Assume that $U\cap (V_0^-\setminus \overline{W_0})\not=\emptyset$ and $U\cap W_0=\emptyset$. 
In this case, it follows from Lemma \ref{lem:topological} that the complement $W_0^-\setminus U$ consists of two connected components. Among them, there uniquely exists a component whose boundary coincides with a connected component of $\del U$, which we denote by $U^c_{\rm outside}$. 
It follows from Lemma \ref{lem:natural} $(ii)$ that there exists a non-constant pluriharmonic function $h_1$ on $W_0^-\setminus U^c_{\rm outside}$ such that $h_1|_{W_0}\in \mathcal{H}^*(\psi_0, W_0, (c_0, d_0))$ and that that $h_1|_U\in \mathcal{H}^*(\psi, U, I)$. 
From Lemma \ref{lem:S42first}, it follows that we may assume that $h_1|_{W_0}=h_0$ by replacing $h_1$ with $c_1h_1+c_2$ for a suitable constants $c_1, c_2\in \mathbb{R}$. 
In this case, we define $h_U\colon W_0^-\to \mathbb{R}\cup\{-\infty\}$ by 
\[
h_U:=\begin{cases}
h_1 & \text{on}\ W_0^-\setminus U^c_{\rm outside}\\
\inf_{W_0^-\setminus U^c_{\rm outside}}h_1 & \text{on}\ U^c_{\rm outside}
\end{cases}. 
\]
Note that, as it is clear by applying the identity theorem to $h_U|_{W_0}$, this construction of $h_U$ depends only on $U$ and is independent on $\psi$. 

\subsection{Step 2}
First, let us check the following: 
\begin{lemma}\label{lem:double_star}
Let $(\psi, U, I)$ be an element of $\mathcal{U}(\theta_0)$. Then the following holds for $r\in (\textstyle\min_{W_0^-}h_U, \textstyle\sup_{W_0^-}h_U)$: \\
$(i)$ $W_0^-\setminus h_U^{-1}(r)$ consists of two connected components $\{h_U<r\}$ and $\{h_U>r\}$. \\
$(ii)$ If $r$ is a regular value of $h_U$, $h_U^{-1}(r)$ is connected. 
\end{lemma}

\begin{proof}
$(i)$ As $\del \{x\in W_0^-\setminus U^c_{\rm outside}\mid h_U(x)<r\} =\del U^c_{\rm outside} \cup h_U^{-1}(r)$ and  
$h_U$ takes the value $\textstyle\min_{W_0^-}h_U$ at any point of $\del U^c_{\rm outside}$ and $r$ at any point of $h_U^{-1}(r)$, 
one has that any connected component of $\{x\in W_0^-\setminus U^c_{\rm outside}\mid h_U(x)<r\}$ touches both $\del U^c_{\rm outside}$ and $h_U^{-1}(r)$ by the maximum principle for the non-constant pluriharmonic function $h_U|_{W_0^-\setminus U^c_{\rm outside}}$, from which one can easily deduce that $\{h_U^{-1}<r\}$ is connected. 
As one can show the connectedness of $\{h_U^{-1}>r\}$ by the same argument, the assertion holds. \\
$(ii)$ 
Take a small positive number $\ve$ such that $(r-\ve, r+\ve)$ is included in the set of all the regular values of $h_U$. By the argument in \S \ref{sec:fund_obs_F-adaptive}, $h_U^{-1}((r-\ve, r+\ve))$ is homeomorphic to the product $h_U^{-1}(r)\times (r-\ve, r+\ve)$. 

First let us consider the case where $r+\ve>\textstyle\inf_{W_0}h_0$. 
In this case, it follows from $\chi_0'>0$ and Ehresmann's lemma for $\psi_0|_{W_0}$ that $A:=h_0^{-1}(r+\ve)\cap W_0$ is connected. 
Let $W$ be the connected component of $W_0^-\setminus A$ such that $\del W=A\cup \del U^c_{\rm outside}$. By the maximum principle for $h_U|_{W}$, one has that $A=h_0^{-1}(r+\ve)$, from which the assertion follows. 

When $r-\ve<\textstyle\sup_{U}h_U$, one can show the assertion by the same argument as in the case where $r+\ve>\textstyle\inf_{W_0}h_0$. 
Therefore, by replacing $\ve$ with a smaller number if necessary, it is sufficient to show the assertion by assuming that $U\subset W_0^-\setminus W_0$ and that $\textstyle\sup_{U}h_U<r-\ve<r+\ve<\textstyle\inf_{W_0}h_0$ (Here recall that $\textstyle\sup_{U}h_U<\textstyle\inf_{W_0}h_0$ holds in this case by the argument in the proof of Lemma \ref{lem:natural}). 
As $h_U^{-1}((r-\ve, r+\ve))$ is homeomorphic to the product $h_U^{-1}(r)\times (r-\ve, r+\ve)$, the assertion is reduced to the connectedness of $h_U^{-1}((r-\ve, r+\ve))$. 
Take a connected component $W$ of $h_U^{-1}((r-\ve, r+\ve))$. 
By Lemma \ref{lem:S42first} and the argument in the proof of Lemma \ref{lem:natural}, one can take constants $c_1, c_2\in\mathbb{R}$, $\theta_\pm\in {\rm SP}(\alpha)$ such that ${\rm Supp}\,\theta_-\subset U$ and ${\rm Supp}\,\theta_+\subset W_0$, and an element $f\in {\rm PSH}^\infty(X, \theta_-)$ such that $h_U=c_1f+c_2$ holds on $W$. 
As $(f, W, f(W))$ is an element of $\mathcal{U}_c(\theta_-)$ by replacing $\ve$ with a smaller number if necessary, it follows from {\bf Condition $(\diamondsuit)$} that the complement $W_0^-\setminus W$ consists of two connected component. 

Assume that $h_U^{-1}((r-\ve, r+\ve))$ has another component $W'$. 
Then it follows from the assertion $(i)$ and a simple topological argument that $\{h_U\not=r\}\cup W'$ is connected, which contradicts to the non-connectedness of $W_0^-\setminus W$. 
Therefore one has that $h_U^{-1}((r-\ve, r+\ve))=W$ is connected. 
\end{proof}

Now let us define a function $h_-\colon W_0^-\to \mathbb{R}\cup\{-\infty\}$ by 
$h_-(x):=\inf\{h_U(x)\mid (\psi, U, I)\in \mathcal{U}(\theta_0)\}$, 
and the set $M^-$ by $M^-:=\textstyle\bigcap\{U^c_{\rm outside}\mid (\psi, U, I)\in \mathcal{U}(\theta_0)\}$ (Though apparently this definition of $M^-$ maybe different from that in \S \ref{sec:5.1}, we will soon show the equivalence between these definitions). 

We show the following: 
\begin{lemma}\label{lem:5.5}
For elements $(\psi, U, I)$ and $(\vp, W, J)$ of $\mathcal{U}(\theta_0)$, the following holds: \\
$(i)$ $W_0^-\setminus M^-$ is connected. \\
$(ii)$ Either $U^c_{\rm outside}\subset W^c_{\rm outside}$ or $W^c_{\rm outside}\subset U^c_{\rm outside}$ holds. \\
$(iii)$ $h_-=h_U$ holds on $W_0^-\setminus U^c_{\rm outside}$. \\
$(iv)$ $h_-|_{W_0^-\setminus M^-}$ is a non-constant pluriharmonic function. \\
$(v)$ $M^-=\{x\in W_0^-\mid h_-(x)=\textstyle\min_{W_0^-}h_-\}$. 
\end{lemma}

\begin{proof}
$(i)$ Note that $W_0^-\setminus U^c_{\rm outside}$ is connected by the definition of $U^c_{\rm outside}$. The assertion holds just by a simple topological argument. \\
$(ii)$ The assertion is clear by Lemma \ref{lem:comparison} when $U\cap W\not=\emptyset$. 
When $U\cap W=\emptyset$, the assertion follows from Lemma \ref{lem:topological}. \\
$(iii)$ It is sufficient to show that $h_U\leq h_W$ holds on $W_0^-\setminus U^c_{\rm outside}$ for an element $(\vp, W, J)\in \mathcal{U}(\theta_0)$. 
As this inequality easily follows from Lemma \ref{lem:comparison} (and the identity theorem for $h_U|_{W_0}$ and $h_W|_{W_0}$) when $U\cap W\not=\emptyset$, in what follows we assume that $U\cap W=\emptyset$. 
By the assertion $(ii)$, Either $U^c_{\rm outside}\subset W^c_{\rm outside}$ or $W^c_{\rm outside}\subset U^c_{\rm outside}$ holds. \\
\ First let us consider the case $W^c_{\rm outside}\subset U^c_{\rm outside}$. 
As both the functions $h_U$ and $h_W$ are pluriharmonic functions on a domain $W_0^-\setminus U^c_{\rm outside}$ which coincide with $h_0$ on $W_0$, it follows from the identity theorem that $h_U=h_W$ on $W_0^-\setminus U^c_{\rm outside}$. \\
\ Next we consider the case $U^c_{\rm outside}\subset W^c_{\rm outside}$. 
From the same argument as above, it follows that $h_U=h_W$ holds on $W_0^-\setminus W^c_{\rm outside}$. 
Therefore, by Lemma \ref{lem:double_star} and the definitions of $h_U$ and $h_W$, $h_W|_{W_0^-\setminus U^c_{\rm outside}}$ coincides with the function $x\mapsto \max\{h_U(x), A\}$, where $A$ is the value of $h_U$ along $\del W^c_{\rm outside}$. 
Therefore $h_U\leq h_W$ holds on $W_0^-\setminus U^c_{\rm outside}$. \\
$(iv)$ Let $x$ be a point of $W_0^-\setminus M^-$. Then, by the definition of $M^-$, $x$ is included in $W_0^-\setminus U^c_{\rm outside}$ for some $(\psi, U, I)\in \mathcal{U}(\theta_0)$. 
As it follows from the assertion $(iii)$ that $h_-$ is a non-constant pluriharmonic function on a neighborhood of $x$, the assertion holds. \\
$(v)$ First we show the inclusion $\{x\in W_0^-\mid h_-(x)=\textstyle\min_{W_0^-}h_-\}\subset M^-$. 
Take a point $x\in W_0^-\setminus M^-$. 
Then there exists $(\psi, U, I)\in \mathcal{U}(\theta_0)$ such that $x\in W_0^-\setminus U^c_{\rm outside}$. 
By the assertion $(iii), (iv)$ and the maximum principle, it holds that $h_-$ does not attain the minimum value at $x$. 
Next we show the opposite inclusion. Take a point $x\in M^-$. 
Fix a point $y\in W_0^-$. 
By the definition of $M^-$ and $h_U$, $h_U(x)\leq h_U(y)$ holds for any $(\psi, U, I)\in \mathcal{U}(\theta_0)$. 
Thus one obtains that $h_-(x)\leq h_-(y)$ holds, from which the assertion follows. 
\end{proof}

\subsection{Step 3}
Construct $h^+\colon W_0^+\to \mathbb{R}\cup\{+\infty\}$ and define $M^+\subset W_0^+$ in the same manner as in Steps 1 and 2 so that $h_+|_{W_0^+\setminus M^+}$ is a pluriharmonic function which coincides with $h_0$ on $W_0$. 
Denote by $M$ the union $M^-\cup M^+$. 
Define a function $h_\alpha\colon X\to [-\infty, +\infty]$ by 
\[
h_\alpha:=\begin{cases}
h_+ & \text{on}\ W_0^+\\
h_- & \text{on}\ W_0^-
\end{cases}. 
\]
First we show the following: 
\begin{lemma}\label{lem:MinK}
The set $M$ is included in $K_\alpha$. 
\end{lemma}

\begin{proof}
We show that any point $x\in M$ is a point of $K_\alpha$. 
We may assume that $x\in M^-$, since the proof for the case $x\in M^+$ can be done by the same argument as in this case. 
Assuming $x\not \in K_\alpha$, we show the assertion by construction. 
By Lemma \ref{lem:K_alpha}, there exists $\psi\in {\rm PSH}^\infty(X, \theta_0)$ such that $(d\psi)_x\not=0$. 
As $r:=\psi(x)$ is neither the maximum or minimum of $\psi$, it follows from Sard's theorem that $r-\ve_1$ and $r+\ve_2$ are regular values of $\psi$ for (generic) small positive numbers $\ve_1$ and $\ve_2$. 
Let $D$ be the connected component of $\psi^{-1}((r-\ve_1, r+\ve_2))$ which includes $x$, 
and $H_1, H_2, \dots, H_N$ be the connected components of the boundary $\del D$. 
As $(d\psi)_x\not=0$, we may assume that $N>1$ by replacing $\ve_1, \ve_2$ with smaller ones if necessary. 
As both $r-\ve_1$ and $r+\ve_2$ are regular values, one can take small open neighborhoods $W_j$ of each $H_j$ for $j=1, 2, \dots, N$ such that $U_j:=W_j\cap D$ satisfies $(\psi, U_j, \psi(U_j))\in\mathcal{U}(\theta_0)$. 
We assume that $x\not\in U_j$ by shrinking $U_j$'s if necessary. 
It follows from the connectedness of $D$ and Lemma \ref{lem:topological} that $N=2$. 
By Lemma \ref{lem:5.5} $(ii)$, we may assume 
$(U_1)^c_{\rm outside}\subset (U_2)^c_{\rm outside}$. 
Again by Lemma \ref{lem:topological}, it follows from $U_1\cup U_2\subset D$ that $(D\setminus (U_1\cup U_2))\cap (U_1)^c_{\rm outside}=\emptyset$, which contradicts to $x\in M^-$. 
\end{proof}

Define a holomorphic foliation $\mathcal{F}_\alpha$ on $X\setminus M$ by letting $T_{\mathcal{F}_\alpha} = (\del h_\alpha)^\perp$. 
Note that, by Lemma \ref{lem:MinK}, $\mathcal{F}_\alpha$ is defined especially on $X\setminus K_\alpha$. 
Note that $h_\alpha|_{\overline{X\setminus M}}$ is clearly an $\mathcal{F}_\alpha$-adaptive function. 
\begin{lemma}\label{lem:F_alpha_III}
The following holds:\\
$(i)$ $K_\alpha\setminus M=\{x\in X\setminus M\mid (dh_\alpha)_x=0\}$. \\
$(ii)$ Any element of ${\rm SP}(\alpha)$ is zero along each leaf of $\mathcal{F}_\alpha|_{X\setminus M}$. 
\end{lemma}

\begin{proof}
We will show the lemma only on $W_0^-$, since one can run the same argument on $W_0^+$. 
Take $(\psi, U, I)\in \mathcal{U}(\theta_0)$. 
Then, as $h_U|_{W_0^-\setminus U^c_{\rm outside}}$ is a function such as $h_W$ in Lemma \ref{lem:natural} $(ii)$ for $W=W_0^-\setminus U^c_{\rm outside}$, the assertions $(i)$ and $(ii)$ on $W_0^-\setminus U^c_{\rm outside}$ (i.e. $K_\alpha\cap (W_0^-\setminus U^c_{\rm outside})=\{x\in W_0^-\setminus U^c_{\rm outside}\mid (dh_\alpha)_x=0\}$ and the $\mathcal{F}_\alpha|_{X\setminus U^c_{\rm outside}}$-leafwise triviality of elements of ${\rm SP}(\alpha)$) follows from Lemma \ref{lem:natural2}, which proves the lemma. 
\end{proof}

Finally, we show the following: 
\begin{lemma}\label{lem:M=Kess}
The set $M$ coincides with $K_\alpha^{\rm ess}$. 
\end{lemma}

\begin{proof}
As $h_\alpha|_{\overline{X\setminus M}}$ is an $\mathcal{F}_\alpha$-adaptive function, $K_\alpha^{\rm ess}\subset M$ holds. 
In what follows, we show that $x\in K_\alpha^{\rm ess}$ for any point $x\in M$. 
Assuming $x\not\in K_\alpha^{\rm ess}$, we will prove it by contradiction. 

As it follows from Lemma \ref{lem:MinK} that $x\in K_\alpha$, this assumption implies that, for the connected component $K'$ of $K_\alpha$ which contains $x$, there exists a connected open neighborhood $W$ of $K'$ in $X$ such that $W\cap K_\alpha$ is a relatively compact subset of $W$ and that there exists an $\mathcal{F}_\alpha$-adaptive function $h\colon \overline{W}\to [-\infty, +\infty]$. 
Let $c:=\textstyle\min_{\overline{W}}h$ and $d:=\textstyle\max_{\overline{W}}h$. 
As $W\cap K_\alpha$ is a relatively compact subset of $W$, $h(W\cap K_\alpha)$ is a relatively compact subset of $(c, d)$. 
Therefore there exist real numbers $a$ and $b$ such that $c<a<b<d$ and that both $h^{-1}((c, a))\cap K_\alpha$ and $h^{-1}((b, d))\cap K_\alpha$ are empty. 
Note that then it follows that $x\in h^{-1}((a, b))$. 

As is followed from the arguments in \S \ref{sec:fund_obs_F-adaptive}, 
$h$ is a pluriharmonic proper submersion on $h^{-1}((a-\ve, a))\cup h^{-1}((b, b+\ve))$ for a small positive number $\ve$. 
Denote by $D_\ve$ the connected component of $h^{-1}((a-\ve, b+\ve))$ which contains $x$. 
Let $H_1, H_2, \dots, H_N$ be the connected components of the boundary $\del D_\ve$. 
By Lemma \ref{lem:M=Kess_sub} below, it follows by replacing $\ve$ with a smaller number that $h_\alpha|_{H_j}\equiv r_j$ for some $r_j\in \mathbb{R}$ for each $j=1, 2, \dots, N$. 
Note that $h_\alpha^{-1}(r_j)=H_j$ holds, since $h_\alpha^{-1}(r_j)$ is connected by 
Lemma \ref{lem:double_star} $(ii)$ and Lemma \ref{lem:5.5} $(iii)$. 

When $N=1$, $\del D_\ve$ is connected and thus $h|_{\del D_\ve}$ is constant, which contradicts to the maximum principle since $h|_{D_\ve}$ is a non-constant pluriharmonic function. 
Thus one has $N>1$. 
It follows from the connectedness of $D_\ve$ and Lemma \ref{lem:double_star} $(i)$ that $N=2$ and $D_\ve\subset h_\alpha^{-1}((r_1, r_2))$. 
Therefore one have that there exists an open neighborhood of $x$ on which $h_\alpha$ is a non-constant pluriharmonic, which contradicts to the assumption $x\in M$ by Lemma \ref{lem:5.5} $(v)$. 
\end{proof}

\begin{lemma}\label{lem:M=Kess_sub}
Let $a, b, \ve$, and $h\colon \overline{W}\to [-\infty, +\infty]$ be as in the proof of Lemma \ref{lem:M=Kess}. 
Then, for any $r\in (a-\ve, a)\cup (b, b+\ve)$ and for any connected component $A$ of $h^{-1}(r)$, $h_\alpha|_A$ is constant. 
\end{lemma}

\begin{proof}
Assume that there exists a connected component $A$ of $h^{-1}(r)$ such that $h_\alpha|_A$ is not constant for some $r\in (a-\ve, a)\cup (b, b+\ve)$. 
As $h$ is $\mathcal{F}_\alpha$-leafwise constant, 
one can run the same argument as in the proof of Lemma \ref{lem:fibration} to obtain a leaf $Y\subset A$ of $\mathcal{F}_\alpha$ and a surjective holomorphic map $\Phi\colon X\to R$ to a compact Riemann surface $R$ such that $Y$ is a fiber of $\Phi$. 
Take a K\"ahler form $\omega_R$ of $R$ and set $\theta:=\Phi^*\omega_R$. 
Then it clearly holds that $d\theta\equiv 0$ and $\theta\wedge \theta\equiv 0$. 
Moreover, as $\{\theta\}=B\cdot c_1([Y])$ holds (since $Y$ is a fiber of $\Phi$) and $\theta_0|_Y\equiv 0$ (by Lemma \ref{lem:F_alpha_III} $(ii)$), it holds that $\{\theta\wedge \theta_0\}=0\in H^{2, 2}(X, \mathbb{C})$. 
Therefore one obtains from Lemma \ref{lem:hodge_index_thm} that we may assume $\theta\in {\rm SP}(\alpha)$ by multiplying some positive constant to $\omega_R$ if necessary. 

Let $\psi\colon R\to \mathbb{R}$ be a function of class $C^\infty$. 
Then $\delta\cdot \psi\in {\rm PSH}^\infty(R, \omega_R)$ holds for sufficiently small positive $\delta$, since $\omega_R>0$. 
Take a regular value $r$ of $\delta\cdot \psi$. 
Then $\vp:=\delta\cdot \psi\circ \Phi$, which is clearly an element of ${\rm PSH}^\infty(X, \theta)$, clearly satisfies the condition that any connected component of $\vp^{-1}(r)$ admits a non-constant $\mathcal{F}(\theta, \vp)$-leafwise contant function of class $C^\infty$, which contradicts to {\bf Condition $(\heartsuit)$}. 
\end{proof}

\subsection{End of the proof}
By the arguments in \S \ref{sec:5.1} and the previous subsection, 
one can construct a foliation $\mathcal{F}_\alpha$ as in Theorem \ref{thm:main1} on the domain $X\setminus K_\alpha^{\rm ess}$ in each of the cases. 

In \S \ref{sec:5.1}, we saw that it is sufficient for proving Theorem \ref{thm:main2} to show the assertion in Case I\!I\!I in this theorem by assuming 
{\bf Condition $(\heartsuit)$} and {\bf Condition $(\diamondsuit)$}, which is done in the previous subsection. 

The uniqueness of $\mathcal{F}_\alpha$ is clear in Case I. 
In the other cases, $\mathcal{F}_\alpha|_W$ is unique for a domain $W$ such that $(\psi, W, \psi(W))\in \mathcal{U}(\theta_0)$ for some $\psi\in {\rm PSH}^\infty(X, \theta_0)$. 
Thus the uniqueness is shown by considering the identity theorem by regarding a holomorphic foliation on $X\setminus K_\alpha^{\rm ess}$ as a holomorphic section of the projective bundle $\mathbb{P}(T_{X\setminus K_\alpha^{\rm ess}})\to X\setminus K_\alpha^{\rm ess}$ by considering the tangent bundle of the foliation, which proves Theorem \ref{thm:main1}. \qed

\section{Semi-positivity of the line bundle associated with an effective divisor with flat normal bundle and Ueda's classification} \label{sec:6}

Let $X$ be a connected compact K\"ahler manifold of dimension $n$, 
and $D$ be an effective divisor of $X$ whose support is smooth or has only simple normal crossing singularities (for simplicity). 
Assume $[D]|_{Z_\lambda}$ is unitary flat line bundle for any $\lambda$, where $D=\textstyle\sum_{\lambda}m_\lambda Z_\lambda$ is the irreducible decomposition ($m_\lambda\in\mathbb{Z}_{>0}$, $Z_\lambda\subset X$ is a reduced irreducible hypersurface). 
Then the class $\alpha:=2\pi c_1([D])$ satisfies ${\rm nd}(\alpha)=1$, since 
\[
(\alpha^2, \{\eta\}) = 2\pi\cdot \sum_{\lambda}m_\lambda\int_{Z_\lambda}\sqrt{-1}\Theta_h\wedge \eta =0
\]
holds for any $d$-closed $(n-2, n-2)$-form $\eta$ of class $C^\infty$, where $h$ is a Hermitian metric of $[D]$ and $\Theta_h$ is the Chern curvature tensor of $h$. 

In what follows, assume that $[D]$ is semi-positive: i.e. ${\rm SP}(\alpha)\not=\emptyset$. 
Take a $C^\infty$ Hermitian metric $h$ on $[D]$ such that $\theta:=\sqrt{-1}\Theta_h\in {\rm SP}(\alpha)$. 
Let us consider the function $F\colon X\setminus |D|\to \mathbb{R}$, where $|D|$ is the support $\textstyle\bigcup_\lambda Z_\lambda$ of $D$, defined by 
$F:=-\log |s_D|_h^2$ for the canonical section $s_D\in H^0(X, [D])$. 
Then, by a simple calculation, one has that 
the function $\psi=\log(1+e^F)-F$ is an element of ${\rm PSH}^\infty(X, \theta)$. 
Note that 
\[
\psi = \log(1+|s_D|_h^2) = |s_D|_h^2 + o(|s_D|_h^2)
\]
holds as a point approaches to $|D|$, from which it follows that $|D|=\{\psi=0\}$ and that, for a sufficiently small positive number $\ve$, any point of $\{0<\psi<\ve\}$ is a regular point of $\psi$. 

By using this function $\psi$, we can show Corollary \ref{cor:main} as follows. 

\begin{proof}[Proof of Corollary \ref{cor:main}]
Consider the case where $D=Y$ is a non-singular connected hypersurface. 
When $N_{Y/X}^m$ is holomorphically trivial for some positive integer $m$, the assertion has already been shown in \cite[Theorem 1.1 $(i)$]{K2020}. 
Assume that $N_{Y/X}^m$ is not holomorphically trivial for any positive integer $m$. 
In this case, by \cite[Theorem 1.4]{K2020}, it is sufficient to show the existence of a connected open neighborhood $\Omega$ of $Y$ such that $\del \Omega$ is a compact Levi-flat hypersurface of class $C^2$. 
When the assertion $(b)$ of Theorem \ref{thm:main_sec3} holds, we can show the existence of $\Omega$ by letting $\Omega:=\{\psi<\ve\}$ for a sufficiently small positive number $\ve$. 
From now on, assuming that the assertion $(a)$ of Theorem \ref{thm:main_sec3} holds, we will prove the assertion by contradiction. 
If $\Phi(Y)=R$, $Y$ intersects all the fibers of $\Phi$, which contradicts to the fact that there exits a fiber of $\Phi$ included in $\{\psi=\ve\}$ for a small positive number $\ve$, which follows from the argument in the proof of Theorem \ref{thm:main_sec3}. 
Thus one has that $\Phi(Y)=\{p\}$ holds for a point $p\in R$. 
Therefore $N_{Y/X}^m$ is holomorphically trivial, where $m$ is the integer such that $\Phi^*\{p\}=mY$ holds as divisors, which contradicts to the assumption. 
\end{proof}

\begin{observation}\label{obs:divisor_main}
Consider the case where $D=Y$ is a non-singular connected hypersurface. 
By Theorem \ref{thm:main_sec3} for $\mathcal(\theta, \psi)$, it follows by the observation above that $\mathcal{F}_\alpha$ is a non-singular holomorphic foliation on $\{0<\psi<\ve\}$. 
Moreover, as is clear by considering the Monge--Amp\`ere foliation for $\sqrt{-1}\ddbar \psi$, $\mathcal{F}_\alpha$ can be $C^\infty$-smoothly extended to $\{\psi<\ve\}$ by adding $Y$ as a leaf. 
Therefore it follows from \cite[Lemma 4.4]{K2020} that $\mathcal{F}_\alpha$ can be extended to $\{\psi<\ve\}$ as a non-singular holomorphic foliation. 
In \cite{K2020}, we considered such a foliation and showed the linearizability of the holonomy along $Y$ by applying P\'erez-Marco's theory \cite{P}. 
We expect that the same argument makes sense even when $|D|$ admits only ``mild" singularities in some sense, however it seems that more precise study on the holonomy along singular leaves is needed for realizing such an argument. 
\end{observation}

Assume that $D=Y$ is a non-singular connected hypersurface. 
Then, if $[Y]$ is semi-positive, $\#{\rm SP}(\alpha)>1$ holds since ${\rm PSH}^\infty(X, \theta)\not=\mathbb{R}$. 
Thus it follows that either of the assertions in Case I, I\!I, or I\!I\!I of Theorem \ref{thm:main2} holds. 
As is seen in the proof of Corollary \ref{cor:main}, $(X, \alpha)$ is in Case I if $N_{Y/X}^m$ is holomorphically trivial for some positive integer $m$ by \cite[Theorem 1.1 $(i)$]{K2020}. 
Let us consider the case where $N_{Y/X}^m$ is not holomorphically trivial for any positive integer $m$. 
In this case, it follows from Corollary \ref{cor:main} that one can choose a Hermitian metric $h$ of $[Y]$ such that $\theta=\sqrt{-1}\Theta_h$ is identically zero on a neighborhood $V$ of $Y$: i.e. $h|_V$ is {\it a flat metric} on $[Y]|_V$. 
In this case, the function $F=-\log|s_D|_h^2$ is pluriharmonic on $V\setminus Y$.  
As it follows from the argument in the proof of Corollary \ref{cor:main} that any $\mathcal{F}_\alpha$-leafwise constant function on a neighborhood of $Y$ is $\psi$-fiberwise constant, 
one has the following by the same argument as in the proof of Lemma \ref{lem:S42first} or Lemma \ref{lem:4.3} $(iii)$: for a sufficiently small number $\ve$, any pluriharmonic function $h$ on $\{0<\psi<\ve\}$ satisfies $h=c_1F|_{\{0<\psi<\ve\}}+c_2$ for constants $c_1, c_2\in \mathbb{R}$. 
Thus $Y\subset K_\alpha^{\rm ess}$, from which it follows that 
$(X, \alpha)$ is in Case I\!I\!I. 


\section{Examples} \label{sec:7}

Here we give some examples. 

\subsection{Non semi-positive case}
The condition ${\rm nd}(\alpha)=1$ does not imply the existence of a semi-positive representative in general. 
Indeed, as is followed from Corollary \ref{cor:main}, for a connected non-singular hypersurface $Y$ of a connected compact K\"ahler manifold $X$ with unitary flat normal bundle, 
$\alpha:=c_1([Y])$ satisfies ${\rm nd}(\alpha)=1$ and ${\rm SP}(\alpha)=\emptyset$ if $[Y]|_V$ is not unitary flat for any neighborhood $V$ of $Y$: i.e. the pair $(Y, X)$ is of class $(\alpha)$ or $(\gamma)$ in Ueda's classification \cite{U}. Serre's example gives a typical concrete example, see also \cite[Example 1.7]{DPS} and \cite{K2015}. 

\subsection{Suspension construction}
Many interesting examples can be constructed by considering the suspension in the following manner. 
Let $Z$ be a connected complex manifold and $F$ be a connected compact Riemann surface. 
Fix a representation $\rho\colon \pi_1(Z, *)\to {\rm Aut}(F)$ of the fundamental group of $Z$, where ${\rm Aut}(F)$ is the group of holomorphic automorphisms of $F$. 
Then one can construct a complex manifold $X$ which has a locally trivial $F$-bundle structure over $Z$ whose monodromy coincides with $\rho$ by letting $X:=F\times \widetilde{Z}/\sim_\rho$, 
where $\widetilde{Z}$ is the universal covering of $Z$ and $\sim_\rho$ is the relation such that $(w, z)\sim_\rho (w', z')$ holds for $w, w'\in F$ and $z, z'\in \widetilde{Z}$ if and only if there exists an element $\gamma\in\pi_1(Z, *)$ such that $w'=\rho(\gamma)(w)$ and $z'=\gamma\cdot z$. 
Assume that there exists a K\"ahler class $\alpha_F$ of $F$ which is invariant by any element of ${\rm Image}\,\rho$. 
Then it is clear that the class ${\rm Pr}_1^*\alpha_F$ of $F\times \widetilde{Z}$ induces a class $\alpha$ of $X$ such that ${\rm nd}(\alpha)=1$, where ${\rm Pr}_1\colon F\times \widetilde{Z}\to F$ is the first projection. 
In this subsection, we give some examples of such $(X, F, Z, \rho, \alpha)$. 

\begin{example}
Consider the case where both $Z$ and $F$ are an elliptic curves. 
For simplicity, we assume that $F\cong\mathbb{C}/\langle 1, \sqrt{-1}\rangle$. 
Consider the representation $\rho\colon \langle 1, \tau\rangle\to {\rm Aut}(F)$ defined by 
$\rho(1)={\rm id}_F$ and $\rho(\tau)$ is the parallel transformation induced by $+(a+b\sqrt{-1})\colon \mathbb{C}\to \mathbb{C}$, where $\tau$ is the modulus of $Z$, ${\rm id}_F$ is the identity map, and $a, b\in \mathbb{R}$. 
Note that $X$ is a complex torus $\mathbb{C}^2_{(w, z)}/\Lambda$ in this case, where $\Lambda$ is the lattice 
\[
\left\langle
 \left(
  \begin{array}{c}
     1 \\
     0
  \end{array}
 \right),\ 
 \left(
  \begin{array}{c}
     \sqrt{-1} \\
     0
  \end{array}
 \right),\ 
 \left(
  \begin{array}{c}
     0 \\
     1
  \end{array}
 \right),\ 
 \left(
  \begin{array}{c}
     a+b\sqrt{-1} \\
     \tau
  \end{array}
 \right)
\right\rangle. 
\]
Let $\theta$ be the form $\sqrt{-1}dw\wedge d\overline{w}$ on $X$ and $\alpha\in H^{1, 1}(X, \mathbb{R})$ be the class which is represented by $\theta$. 
Let $\mathcal{F}_\alpha$ be the foliation by curves which is determined by the eigenvectors which belongs to the eigenvalue zero of $\theta$. 
As is easily seen, each leaf of $\mathcal{F}_\alpha$ is locally defined by $\{w=\text{constant}\}$, from which it follows that $\mathcal{F}_\alpha$ is a non-singular holomorphic foliation. 

First, let us consider the case where both $a$ and $b$ are rational. 
In this case, $R:=\mathbb{C}_w/\langle 1, \sqrt{-1}, a+b\sqrt{-1}\rangle$ is an elliptic curve. 
By considering the morphism $\Phi\colon X\to R$ which is induced by the projection $\mathbb{C}^2_{(w, z)}\to \mathbb{C}_w$, it follows that $\#{\rm SP}(\alpha)>1$ and $(X, \alpha)$ is in Case I of Theorem \ref{thm:main2} in this case. 

Next, let us consider the case where $a$ or $b$ is irrational and $a = q b$ or $b = q a$ holds for some rational number $q$. 
For simplicity, here we restrict ourselves to the case $a=0$ and $b\in\mathbb{R}\setminus \mathbb{Q}$. 
In this case, consider the function $\pi\colon X\to \mathbb{R}/\mathbb{Z}$ induced by the function $\mathbb{C}^2_{(w, z)}\ni (w, z)\mapsto {\rm Re}\,w\in \mathbb{R}$. 
Then, for any $t\in \mathbb{R}/\mathbb{Z}$, the fiber $H_t:=\pi^{-1}(t)$ is a real analytic compact Levi-flat hypersurface of $X$ such that any leaf $\mathcal{L}\subset H_t$ of $\mathcal{F}_\alpha$ is dense in $H_t$. Note that any $\mathcal{F}_\alpha$-leafwise constant continuous function on $H_t$ is constant in this case. 
Take a non-constant function $\psi\colon \mathbb{R}/\mathbb{Z}\to \mathbb{R}$ of class $C^\infty$. 
Then $\vp:=\ve\cdot (\psi\circ \pi)$ is an element of ${\rm PSH}^\infty(X, \theta)$ for a sufficiently small positive number $\ve$, from which it follows that $\#{\rm SP}(\alpha)>1$. 
As is followed from $K_\alpha=\emptyset$ (or the fact that the complement of any $H_t$ is connected), $(X, \alpha)$ is in Case I\!I of Theorem \ref{thm:main2} in this case. 

Finally, let us consider the case where $a = r b$ or $b = r a$ holds for some irrational number $r$. 
As each leaf of $\mathcal{F}_\alpha$ is dense in $X$, it follows that ${\rm SP}(\alpha)=\{\theta\}$ in this case. \qed
\end{example}

\begin{example}
Consider the case where $F$ is the projective line $\mathbb{P}^1$. 
Here let us consider one of the simplest cases: 
The case where $Z\cong \mathbb{C}/\langle 1, \tau\rangle$ is an elliptic curve and 
the representation $\rho\colon \langle 1, \tau\rangle\to {\rm Aut}(\mathbb{P}^1)$ is the one which is defined by 
$\rho(1)={\rm id}_{\mathbb{P}^1}$ and $\rho(\tau)=U$ for some unitary matrix $U$. 
Note that $X$ is the ruled surface $\mathbb{P}(E_\rho)$ in this case, where $E_\rho$ is the unitary flat vector bundle on $Z$ of rank $2$ which corresponds to the representation $\rho$. 
In this case, consider the fiberwise Fubini--Study form 
$\theta:=\textstyle\frac{\sqrt{-1}}{2\pi}\ddbar\log(|x|^2+|y|^2)$ 
on $X$, where $[x; y]$ is a homogeneous coordinate of a fiber. 
Note that $\theta$ represents the first Chern class $\alpha$ of the relative $\mathcal{O}_{\mathbb{P}^1}(1)$-bundle. 
Let $\mathcal{F}_\alpha$ be the foliation by curves which is determined by the eigenvectors which belongs to the eigenvalue zero of $\theta$. 
As is easily observed, each leaf of $\mathcal{F}_\alpha$ is locally defined by $\{[x; y]=\text{constant}\}$. 
Therefore $\mathcal{F}_\alpha$ is a non-singular holomorphic foliation. 
By choosing $(x, y)$ suitably, we may assume that $U$ corresponds to the unitary rotation $w\mapsto \lambda\cdot w$ for some $\lambda\in {\rm U}(1)$, where $w:=x/y$ is the non-homogeneous coordinate. 

When $\lambda^m=1$ for some positive integer $m$, $\mathcal{F}_\alpha$ coincides with the foliation associated with the fibration $\Phi\colon X\to \mathbb{P}^1$ which is defined by $\Phi(w, z):=w^m$. Thus $(X, \alpha)$ is in Case I. 

When $\lambda^m\not=1$ for any positive integer $m$, consider the zero section $Y_0$ defined by $\{w=0\}$ and the $\infty$-section $Y_\infty$ defined by $\{w=\infty\}$ of the ruled surface $X$. 
As neither $N_{Y_0/X}^m$ nor $N_{Y_\infty/X}^m$ is holomorphically trivial for any positive integer $m$, it follows from the observation in the previous section that $(X, \alpha)$ is in Case I\!I\!I. 
In this example, it follows by considering the function $h_\alpha\colon X\to [-\infty, +\infty]$ defined by $h_\alpha(w, z):=\log|w|$ that $K_\alpha=K_\alpha^{\rm ess}=Y_0\cup Y_\infty$ holds. 
\qed
\end{example}

\subsection{The blow-up of $\mathbb{P}^2$ at nine points and K3 surfaces constructed by gluing}

Here we give an example of a compact K\"ahler surface $X$ and a class $\alpha\in H^{1, 1}(X, \mathbb{R})$ with ${\rm nd}(\alpha)=1$ and $\#{\rm SP}(\alpha)>1$ which is in Case I\!I\!I such that the foliation $\mathcal{F}_\alpha$ on $X\setminus K_\alpha^{\rm ess}$ never can be holomorphically extended to $X$. 

Let $C$ be an smooth cubic of a projective plane $\mathbb{P}^2$. 
Take nine points $p_1, p_2, \dots, p_9\in C$. 
Denote by $X$ the blow-up of $\mathbb{P}^2$ at these nine points, 
and by $Y$ the strict transform of $C$. 
Note that $[-Y]$ coincides with the canonical bundle $K_X$ in this example. 
Set $\alpha:=c_1([Y])$. 
Note that, if $N_{Y/X}^m$ is holomorphically trivial for some positive integer $m$, it is classically known that $X$ admits a structure as an elliptic surface which has $Y$ as a fiber, from which it is clear that $(X, \alpha)$ is in Case I. 

In what follows, we consider the case where $N_{Y/X}^m$ is not holomorphically trivial for any positive integer $m$. 
It is known that ${\rm SP}(\alpha)\not=\emptyset$ for almost all choice of nine points $(p_1, p_2, \dots, p_9)\in C^9$ in the measure sense (Under {\it the Diophantine condition}, see \cite{A}, \cite{B}, \cite{U}, see also \cite{K2019}. Note that we may assume that $X\setminus Y$ includes no compact curve by choosing generic nine points. Note also that Theorem \ref{thm:main3} is a generalization of Brunella's theorem \cite[Theorem 1.1 $(i)$]{B} for such an example). 
As we have seen in the previous section, $(X, \alpha)$ is in Case I\!I\!I and $Y\subset K_\alpha^{\rm ess}$ in this case. 
Moreover, by the argument we mentioned in Observation \ref{obs:divisor_main}, it follows that the foliation $\mathcal{F}_\alpha$ can be holomorphically extended to $(X\setminus K_\alpha^{\rm ess})\cup Y$ by adding $Y$ as a leaf. 
By \cite[Proposition 8]{B}, such a holomorphic foliation never can be holomorphically extended to whole $X$. 

Let $(X, \alpha)$ be as above. 
By the argument in the last of the previous section, 
we have that $h_\alpha=-\log |s_Y|_h^2$ holds on a neighborhood $V$ of $Y$ by replacing $h_\alpha$ with $c_1h_\alpha+c_2$ for some $c_1, c_2\in \mathbb{R}$, where $s_Y\in H^0(X, [Y])$ is the canonical section and $h$ is a flat metric of $[Y]|_V$, whose existence is assured by \cite[Theorem 1.1 $(i)$]{B} (or Theorem \ref{thm:main3}). 
Denote by $M^-$ the complement $K_\alpha^{\rm ess}\setminus Y$. 
Note that that $h_\alpha(x)\to \textstyle\inf_{X\setminus K_\alpha^{\rm ess}}h_\alpha$ holds as $x\to \del M^-$, since $h_\alpha$ is $\mathcal{F}_\alpha$-adaptive. 
The following question is one of the biggest motivation of the present paper: 
\begin{question}
Does it hold that $\textstyle\inf_{X\setminus K_\alpha^{\rm ess}}h_\alpha=-\infty$? 
\end{question}
If $\textstyle\inf_{X\setminus K_\alpha^{\rm ess}}h_\alpha=-\infty$, 
the function $h_\alpha$ can be extended to $X\setminus Y$ by letting $h_\alpha|_{M^-}\equiv -\infty$. In this case, one has that $M^-$ is a pluripolar set since $M^-=\{x\in X\setminus Y\mid h_\alpha(x)=-\infty\}$ and $h_\alpha=\textstyle\inf_{A\in \mathbb{R}} h_A$ is a plurisubharmonic function on $X\setminus Y$, where $h_A$ is the plurisubharmonic function on $X\setminus Y$ defined by $h_A(x):=\max\{h_\alpha(x), A\}$ for $A\in \mathbb{R}$. 

Note that one can also construct an example of $(X, \alpha)$ in Case I\!I\!I such that $K_\alpha^{\rm ess}$ includes no compact curve 
by considering the gluing construction of a K3 surface \cite{KU}. 
Again, let $C$ be an smooth cubic of a projective plane $\mathbb{P}^2$. 
Take eighteen points $p_1^\pm, p_2^\pm, \dots, p_9^\pm\in C$. 
Denote by $X^+$ the blow-up of $\mathbb{P}^2$ at the nine points $p_1^+, p_2^+, \dots, p_9^+$, 
and by $X^-$ the blow-up at $p_1^-, p_2^-, \dots, p_9^-$. 
Denote by $Y^\pm$ the strict transforms of $C$ in $X^\pm$. 
Fix an isomorphism $g\colon Y^+\to Y^-$. 
We may assume that neither of $X^\pm\setminus Y^\pm$ admits compact curve and both $[Y^\pm]$ are semi-positive by choosing eighteen points generically. 
In \cite{KU}, we showed by using \cite{A} that there exist neighborhoods $W^\pm$ of $Y^\pm$ in $X^\pm$ such that one can construct a K3 surface $\widetilde{X}$ by holomorphically gluing $X^+\setminus W^+$ and $X^-\setminus W^-$. 
Take an element $\theta\in {\rm SP}(c_1([Y^+]))$ whose support is included in $U\cap (X^+\setminus W^+)$ for a sufficiently small neighborhood $U$ of $\del W^+$ in $X^+$. 
Denote also by $\theta$ the semi-positive $(1, 1)$-form on $\widetilde{X}$ which coincides with $\theta$ on $X^+\setminus W^+$ and is zero on the complement of $X^+\setminus W^+$. 
Then, for the class $\alpha:=\{\theta\}$ of $\widetilde{X}$, it follows from the arguments in the construction of $h_\alpha$ that $K_\alpha^{\rm ess}$ is the union of $M^-$'s of $(X^\pm, c_1([Y^\pm]))$. 


\end{document}